\documentclass[reqno]{amsart}
\usepackage{amsmath,amsthm,amssymb,amscd,comment,mathrsfs}

\newtheorem{theorem}{Theorem}[section]
\newtheorem{lemma}[theorem]{Lemma}
\newtheorem{proposition}[theorem]{Proposition}

\newtheorem{corollary}[theorem]{Corollary}
\theoremstyle{definition}
\newtheorem{definition}[theorem]{Definition}

\theoremstyle{remark}

\numberwithin{equation}{section}
\allowdisplaybreaks

\begin{document}
	\title[Dirichlet characters and Peisert-like graphs on $\mathbb{Z}_n$]
	{Hypergeometric functions for Dirichlet characters and Peisert-like graphs on $\mathbb{Z}_n$}
	
	
	\author{Anwita Bhowmik}
	\address{Department of Mathematics, Indian Institute of Technology Guwahati, North Guwahati, Guwahati-781039, Assam, INDIA}
	\email{anwita@iitg.ac.in}
	\author{Rupam Barman}
	\address{Department of Mathematics, Indian Institute of Technology Guwahati, North Guwahati, Guwahati-781039, Assam, INDIA}
	\email{rupam@iitg.ac.in}
	

	\date{September 30, 2023: Journal: La Matematica (accepted for publication)}
	\begin{abstract}
		For a prime $p\equiv 3\pmod 4$ and a positive integer $t$, let $q=p^{2t}$. The Peisert graph of order $q$ is the graph with vertex set $\mathbb{F}_q$ such that $ab$ is an edge if $a-b\in\langle g^4\rangle\cup g\langle g^4\rangle$, where $g$ is a primitive element of $\mathbb{F}_q$. In this paper, we construct a similar graph with vertex set as the commutative ring $\mathbb{Z}_n$ for suitable $n$, which we call \textit{Peisert-like} graph and denote by $G^\ast(n)$. Owing to the need for cyclicity of the group of units of $\mathbb{Z}_n$, we consider $n=p^\alpha$ or $2p^\alpha$, where $p\equiv 1\pmod 4$ is a prime and $\alpha$ is a positive integer. For primes $p\equiv 1\pmod 8$, we compute the number of triangles in the graph $G^\ast(p^{\alpha})$ by evaluating certain character sums. Next, we study cliques of order 4 in $G^\ast(p^{\alpha})$. To find the number of cliques of order $4$ in $G^\ast(p^{\alpha})$, we first introduce hypergeometric functions containing Dirichlet characters as arguments, and then express the number of cliques of order $4$ in $G^\ast(p^{\alpha})$ in terms of these hypergeometric functions.
	\end{abstract} 
	\maketitle
	\section{Introduction}
	Let $\mathbb{F}_q$ be the finite field with $q$ elements, where $q$ is a prime power and $q\equiv 1\pmod 4$. The Paley graph of order $q$ is the graph with vertex set $\mathbb{F}_q$ and edges defined as, $xy$ is an edge if $x-y$ is a nonzero square in $\mathbb{F}_q$. The Paley graphs are a well-known family of self-complementary and symmetric graphs.   In 2001, Peisert \cite{peisert} discovered a new infinite family of self-complementary and symmetric graphs which he called $\mathscr{P}^\ast$ graphs (now called \textit{Peisert graphs}) and deduced that every self-complementary and symmetric graph is isomorphic to either a Paley graph or a $\mathscr{P}^\ast $ graph or one exceptional graph with $529$ vertices. To define the Peisert graph, we take $q=p^{2t}$ where $p\equiv 3\pmod 4$ is a prime and $t$ is a positive integer. Let $g$ be a primitive element of the finite field $\mathbb{F}_q$, 
	that is, $\mathbb{F}_q^\ast=\mathbb{F}_q\setminus\{0\}=\langle g\rangle$. Then the Peisert graph of order $q$ is the graph $\mathscr{P}^\ast(q)=(V,E)$, where $V=\mathbb{F}_q$ and $E=\{xy| x-y\in\langle g^4\rangle \cup g\langle g^4\rangle\}$. It turns out that an edge is well defined, since $q\equiv 1\pmod 8$ implies that $-1\in\langle g^4\rangle$. Note that for $q\equiv 1\pmod 4$ where $q$ is an even power of a prime $p\equiv 3\pmod 4,$ the edges of the Paley graph are determined by the cosets $\langle g^4\rangle\cup g^2\langle g^4\rangle$, while those of the Peisert graph depend on the cosets $\langle g^4\rangle\cup g\langle g^4\rangle$. Like the Paley graph, $\mathscr{P}^\ast(q)$ is strongly regular with parameters $\left(q,\frac{q-1}{2},\frac{q-5}{4},\frac{q-1}{4}\right)$.
	Various properties of Peisert graphs have been studied, for example, their automorphism groups by Peisert himself in \cite{peisert}, pseudo-random properties in \cite{pseudo}, structure of maximal and maximum cliques in \cite{yip_maximal} and \cite{rigidity}, critical groups of the graphs in \cite{critical}, etc. Peisert graphs have been used to produce binary and ternary codes from their adjacency matrices in \cite{key2019special}. In \cite{james}, certain designs have been produced from Peisert graphs as well.
	\par 
	Number theorists have introduced finite field hypergeometric functions as generalizations of classical hypergeometric functions by using Gauss and Jacobi sums, see for example \cite{greene,greene2,mccarthy3}. Some of the biggest motivations for studying finite field hypergeometric functions have been their connections with Fourier coefficients and eigenvalues of modular forms and with counting points on certain kinds of algebraic varieties. For example, Ono \cite{ono, ono2} gave formulae for the number of $\mathbb{F}_p$-points on elliptic curves in terms of special values of finite field hypergeometric functions. These functions have recently led to applications in graph theory as well, for example in the study of Paley and Peisert graphs \cite{bhowmikpeisert, dawsey, wage}. 
	\par 
	Besides Paley graphs being generalized (see for example \cite{bhowmik2020paley, lim2006generalised}), Peisert graphs have also been generalized into graphs called \textit{generalized Peisert} or \textit{Peisert type} graphs and the structure of their maximum and maximal cliques have been studied in \cite{psepal, asgarli, rigidity, mullin}. The generalizations of Paley and Peisert graphs are all cyclotomic graphs, that is, Cayley graphs with the connection set being the union of cyclotomic classes (a reference for cyclotomic graphs is \cite{brow}). Cyclotomic graphs are of special interest in both algebraic graph theory and number theory; in particular, character sums and Gauss sums have been extensively used to study such graphs.
	In this article, we introduce a \textit{Peisert-like} graph on the commutative ring $\mathbb{Z}_n$, for suitable $n$. Computing the number of cliques in Paley, Peisert and Paley-type graphs has been of interest, for instance see \cite{james,bhowmik2020paley, bhowmikpeisert, evans1981number, gory}. Our primary focus is to evaluate the number of cliques of orders 3 and 4 in the Peisert-like graph by evaluating certain character sums involving Dirichlet characters. To this end, we introduce hypergeometric functions containing Dirichlet characters as arguments in Section \ref{sechyp}, and then use these functions to compute the number of cliques of order 4 in the Peisert-like graph.
	\section{Defining the Peisert-like graph}
	The analogue of a Paley graph, called the Paley-type graph, was constructed by us in \cite{bhowmik2020paley}, where the vertex set was taken to be the commutative ring $\mathbb{Z}_n$ for suitable $n$.	Let $\mathbb{Z}_n^\ast$ denote the multiplicative group of units of $\mathbb{Z}_n$. We consider $n$ such that $\mathbb{Z}_n^\ast$ is cyclic, omitting the trivial cases $n=2,4$. By a famous result due to Gauss, we have $n=p^\alpha$ or $n=2p^\alpha$, where $p$ is an odd prime and $\alpha$ is a positive integer.  First, we consider the possibilities of constructing graphs by considering two cosets out of the four cosets of the subgroup $\langle g^4\rangle$ in $\mathbb{Z}_n^\ast=\langle g\rangle$. Let $|x|$ denote the order of an element $x\in\mathbb{Z}_n^\ast$. Then,
	\begin{equation*}
		|g^4|=\frac{|g|}{\gcd(4,|g|)}=\frac{p^{\alpha-1}(p-1)}{\gcd(4,p^{\alpha-1}(p-1))}=\frac{p^{\alpha-1}(p-1) }{2 \gcd(2,\frac{p-1}{2})}.
	\end{equation*} 
	If $\gcd(2,\frac{p-1}{2})=1$ then $|g^4|=\frac{|g|}{2}$ and so $\langle g^4\rangle $ has two distinct cosets in $\langle g\rangle $, whereby $\mathbb{Z}_n^\ast$ becomes the union of the two distinct cosets. So, in order that the edge set of the graph we construct depends on a proper subset of $\mathbb{Z}_n^\ast$, we need that $\gcd(2,\frac{p-1}{2})\neq 1$, and hence $2\mid \frac{p-1}{2}$, that is, $p\equiv 1\pmod 4$. Then there are four distinct cosets of $\langle g^4\rangle$ in $\langle g\rangle $. Subsequently, we assume that $n=p^\alpha$ or $n=2p^\alpha$ where $p$ is an odd prime such that $p\equiv 1\pmod 4$ and $\alpha\geq 1$. We look at the possible pairs of distinct cosets of $\langle g^4\rangle $ that can be taken to construct the edge set of a well defined graph. Let the cosets be $g^i\langle g^4\rangle $ and $g^j\langle g^4\rangle , i\neq j$ and $i,j\in \{0,1,2,3\}$. To ensure that an edge is well defined for an undirected graph, we need the property that,
	for $x\in\mathbb{Z}_n$, if $x\in g^i\langle g^4\rangle \cup ~ g^j\langle g^4\rangle$  then $-x\in g^i\langle g^4\rangle \cup ~ g^j\langle g^4\rangle$. Let $G_{i,j}$ denote the graph constructed, if possible, by taking the vertex set to be $\mathbb{Z}_n$, where $G_{i,j}$ has an edge $xy$ if $x-y\in g^i\langle g^4\rangle \cup g^j\langle g^4\rangle $. We find that it is enough to study the case when $i=0$ or $j=0$, for otherwise we can find an isomorphism between $G_{i,j}$ and $G_{0,j-i}$; consequently we consider the possible graphs $G_{0,a}$ where $a\in\{1,2,3\}$. If $a=2$ then we get back the Paley-type graph which we studied in \cite{bhowmik2020paley}, and so we do not consider $G_{0,2}$. Moreover, $G_{0,1}$ and $G_{0,3}$, if well defined, are isomorphic. Thus, the only graph needed to be considered is $G_{0,1}$. For edges in the graph to be well defined, we require that $p\equiv 1\pmod 8$, and hence, we have the following definition.
	\begin{definition}[Peisert-like graph $G^\ast(n)$]\label{defn1}
		Let $n=p^\alpha$ or $n=2p^\alpha$, where $p$ is an odd prime such that $p\equiv 1\pmod 8$ and $\alpha$ is a positive integer. Let $\mathbb{Z}_n^\ast=\langle g\rangle $. Then, the \textit{Peisert-like graph} is the graph $G^\ast(n)=(V,E)$, where $V=\mathbb{Z}_n$ and $E=\{xy| x-y\in \langle g^4\rangle \cup g\langle g^4\rangle \}$.
	\end{definition}
	The definition of the graph is independent of the choice of the generator $g$, like in the Peisert graph. To see this, let $h$ be another generator of $\mathbb{Z}_n^\ast$. Then $h=g^t$ for some $t\in\mathbb{Z}$. If $t$ is even then $h=(g^{\frac{t}{2}})^2\in (\mathbb{Z}_n^\ast)^2$, which implies $\mathbb{Z}_n^\ast\subseteq (\mathbb{Z}_n^\ast)^2 $, which is not possible. So, $t\equiv 1$ or $3\pmod 4$. If $t\equiv 1\pmod 4$, then $\langle g^4\rangle =\langle h^4\rangle $ since both are subgroups of order $\frac{p^{\alpha-1}(p-1)}{4}$ and $\mathbb{Z}_n^\ast$ is cyclic, and $h\langle h^4\rangle =g^t\langle g^4\rangle =g\langle g^4\rangle $. So the edge set remains unchanged. If $t\equiv 3\pmod 4$, we define the graph $G'(n)$ as $G'(n)=(V',E')$, where $V'=\mathbb{Z}_n$ and $E'=\{xy|x-y\in \langle h^4\rangle \cup h\langle h^4\rangle\}$. Then,
	\begin{align*}
		V(G^\ast(n))&\rightarrow V(G'(n))\\
		x&\mapsto hx
	\end{align*}	
	is an isomorphism.
	\section{Statement of main results}\label{secmain}
	The Peisert-like graph $G^\ast(n)$ is defined for $n=p^\alpha$ or $n=2p^\alpha$, where $p\equiv 1\pmod 8$ is a prime and $\alpha$ is a positive integer. However, to calculate the number of cliques of orders three and four in the graph, we omit the case $n=2p^\alpha$. This is because there cannot exist cliques of order more than two if $n=2p^\alpha$, and we see why. Let $n=2p^\alpha$, and if possible let $x,y$ and $z$ be vertices in $G^\ast(n)$ which form a clique. Then $x-y,y-z$ and $x-z$ are necessarily elements in $\mathbb{Z}_n^\ast$, and therefore, are odd integers, which contradicts that $x-z=x-y+y-z$. Thus, we consider only the case $n=p^\alpha$.
	\par Let $k_m(G)$ denote the number of cliques of order $m$ in the graph $G$. In the following theorem, we compute the number of cliques of order three in the Peisert-like graph. 
	\begin{theorem}\label{thm1}
		Let  $p\equiv 1\pmod 8$ be a prime and let $\alpha$ be a positive integer. Let $G^\ast(p^\alpha)$ be the Peisert-like graph of order $p^\alpha$. Then, $$k_3(G^\ast(p^\alpha))=\frac{p^{3\alpha-2}(p-1)(p-5)}{48}.$$
	\end{theorem}
	For a prime $p\equiv 1 \pmod 8$ and a positive integer $\alpha$, we observe that the number of cliques of order three in the Peisert-like graph of order $p^\alpha$ equals the number of cliques of order three in the Paley-type graph $G_{p^\alpha}$ of order $p^\alpha$, introduced in \cite{bhowmik2020paley}.
	\par To find the number of cliques of order $4$ in the Peisert-like graph  $G^\ast(p^\alpha)$, one needs to compute certain character sums involving Dirichlet characters modulo $p^{\alpha}$. We simplify such character sums by introducing hypergeometric functions for Dirichlet characters. Let $n\in\mathbb{Z}$. A completely multiplicative function $\psi:\mathbb{Z}\rightarrow\mathbb{C}$ is called a Dirichlet character modulo $n$ if $\psi(1)=1,\psi(a)=0$ if $\gcd (a,n)>1$ and $\psi(a)=\psi(b)$ if $a\equiv b\pmod n$. For $a\in\mathbb{Z}$, we define $\overline{\psi}(a):=\overline{\psi(a)}$, whence $\overline{\psi}$ becomes a Dirichlet character mod $n$. The set of Dirichlet characters mod $n$ forms a group (denoted by $\widehat{\mathbb{Z}_n^\ast}$) under multiplication defined as $\psi\lambda(a):=\psi(a)\lambda(a)$, where $\psi$ and $\lambda$ are characters mod $n$. For Dirichlet characters $A$ and $B$ modulo $n$, the Jacobi sum is defined as 
	\begin{align*}
		J(A,B):=\sum_{x\in\mathbb{Z}_{n}} A(x)B(1-x).
	\end{align*}
	Analogous to Greene's hypergeometric functions over finite fields \cite{greene,greene2}, we introduce hypergeometric functions for Dirichlet characters in Section \ref{sechyp}. In the following theorem, we find the number of cliques of order $4$ in the Peisert-like graph by using the hypergeometric functions for Dirichlet characters. We denote by $Im(z)$ the imaginary part of the complex number $z$.
	\begin{theorem}\label{thm2}
		Let $q=p^{\alpha}$, where $p\equiv 1 \pmod 8$ is a prime and $\alpha$ is a positive integer. Let $G^\ast(q)$ be the Peisert-like graph of order $q$. Let $\chi_4$ be a character mod $q$ of order $4$, and let $\varphi$ and $\varepsilon$ be the quadratic and trivial characters mod $q$, respectively.  Then, 
		\begin{align*}
			k_4(G^\ast(q))=\frac{p^{2\alpha-1}(p-1)}{3072}[ &2 p^{2\alpha-2}(p^2-20p+81)+2\text{Im}(\rho)^2+4 \text{Im}(\rho)\cdot \text{Im}(\xi)\\
			& -Re(M_3)+3M_5],			
		\end{align*}
		where $\rho:=J(\chi_4,\chi_4)$ and $\xi:=J(\chi_4,\varphi)$; and 
		$M_3= q^2\cdot{_{3}}F_{2}\left(\begin{array}{ccc}\chi_4, & \overline{\chi_4}, & \overline{\chi_4}\\ & \varphi, & \varepsilon\end{array}| 1\right)
		$ and $M_5= q^2\cdot{_{3}}F_{2}\left(\begin{array}{ccc}\chi_4, & \chi_4, & \overline{\chi_4}\\ & \varepsilon, & \varepsilon\end{array}| 1\right)$ are the hypergeometric terms as defined in Section 5. 
	\end{theorem}
	It is evident from the theorem that $M_5$ is a real number, since $k_4(G^\ast(q))$ is a real number. Using Python, we numerically verify Theorem $\ref{thm2}$ for certain values of $p$ and $\alpha$. We list some of the values in Table \ref{Table-1}. We find that for each of the values of $p^\alpha$ listed below, $\rho=\xi$. 
	\begin{table}[!hb]
		\begin{center}
			\begin{tabular}{|c |c | c| c | c| c|}
				\hline
				$q=p^\alpha$ & $\rho=\xi$ & ${M_3}$ & $M_5$ & $k_4(G^\ast(q))$& $k_4(G_q)$\\
				\hline 
				$17^1=17$ & $-1+4i$& $-6-24i$ &$10$ & $17$ &$0$\\
				$41^1=41$ & $-5+4i$ & $-30-24i$ &$-30$ & $1025$ &$1025$\\
				$73^1=73$ & $3+8i$ & $-6+16i$ &$10$ & $14235$ &$13140$\\
				$89^1=89$ & $-5+8i$ & $90+144i$ &$-22$ & $32307$&$31328$ \\
				$97^1=97$ & $-9-4i$ & $90-40i$ &$-150$ & $44426$&$46560$ \\
				$17^2=289$  & $-17+68i$ & $-1734-6936i$ &$2890$ & $1419857$ &$0$\\
				\hline 			
			\end{tabular}
			\caption{Numerical data for Theorem \ref{thm2}}
			\label{Table-1}
		\end{center}
	\end{table}
	The GitHub link for the Python code that we used to compute $k_4
	(G^\ast(q))$ is provided in the appendix. Now, let $G_q$ denote the Paley-type graph defined in \cite{bhowmik2020paley}. We also note that in general, the values of $k_4(G^\ast(q))$ and $k_4(G_q)$ differ. In the last column of Table \ref{Table-1}, using \cite[Theorem 1.2]{bhowmik2020paley} we provide values of $k_4(G_q)$ for a comparison between $k_4(G^\ast(q))$ and $k_4(G_q)$.
	\section{Some properties of the graph and proof of Theorem \ref{thm1}}
	Let $n=p^\alpha$ or $2p^\alpha$, where $p\equiv 1\pmod 8$ and $\alpha\geq 1$, and let $G^\ast(n)$ be the Peisert-like graph of order $n$. Since $4$ divides the order of $\mathbb{Z}_n^\ast$ and $\mathbb{Z}_n^\ast$ is cyclic, there exists a character of order $4$ in $\widehat{\mathbb{Z}_n^\ast}$; let us fix such a character and call it $\chi_4$. Let $\varphi=\chi_4^2$ be the quadratic character. Let $\varepsilon$ denote the trivial character defined by
	\begin{align*}
		\varepsilon(x) = \left\{ \begin{array}{cl} 1, &\text{if } x\in \mathbb{Z}_{n}^{*};\\
			0, & \text{otherwise. }\end{array}\right. 
	\end{align*}
	Let $\mathbb{Z}_n^\ast=\langle g\rangle $ and let $h=1-\chi_4(g)$. Then, for $x\in\mathbb{Z}_n^\ast$, we observe that
	\begin{align}\label{war}
		\frac{2+h\chi_4(x)+\overline{h}\overline{\chi_4}(x)}{4} = \left\{
		\begin{array}{lll}
			1, & \hbox{if $x\in \langle g^4\rangle \cup g\langle g^4\rangle $}; \\
			0, & \hbox{\text{otherwise.}}
		\end{array}
		\right.
	\end{align} 
	Now, we prove some basic properties of $G^\ast(n)$. Let $\phi$ denote the Euler totient function.
	\begin{proposition}
		Let $n=p^\alpha$ or $2p^\alpha$, where $p\equiv 1\pmod 8$ and $\alpha$ is a positive integer. Let $G^\ast(n)$ be the Peisert-like graph of order $n$. Then, $G^\ast(n)$ is regular of degree $\frac{p^{\alpha-1}(p-1)}{2}$. Also, the number of edges in $G^\ast(n)$ is equal to $\frac{n\phi(n)}{4}$.
	\end{proposition}
	\begin{proof}
		Let $\mathbb{Z}_n^\ast=\langle g\rangle$. By the definition of $G^\ast(n)$, the degree of a vertex is equal to the cardinality of the set $\langle g^4\rangle\cup g\langle g^4\rangle$. Alternatively, we may use a character sum to deduce the same.
		Let $a\in\mathbb{Z}_n $. Then, using \eqref{war}, we find that the degree of the vertex $a$ is
		\begin{align*}
			deg(a)=\sum\limits_{x-a\in\mathbb{Z}_n^\ast}\dfrac{2+h\chi_4(a-x)+\overline{h}\overline{\chi_4}(a-x)}{4}=\dfrac{p^{\alpha-1}(p-1)}{2}.
		\end{align*}
		The last equality is obtained by using  $\sum\limits_{x-a\in\mathbb{Z}_n^\ast}\chi_4(a-x)=\sum\limits_{x-a\in\mathbb{Z}_n^\ast}\overline{\chi_4}(a-x)=0$. 
		The number of edges in $G^\ast(n)$ is $\frac{1}{2}\times \sum \text{deg}=\frac{1}{2}\frac{p^{\alpha-1}(p-1)}{2}\times n=\frac{n\phi(n)}{4}$.
		This completes the proof of the proposition.
	\end{proof}
	Alternatively, one can find the number of edges in $G^\ast(n)$ by evaluating the following character sum: \begin{align*}
		\frac{1}{2} \sum\limits_x \sum\limits_{y-x\in\mathbb{Z}_n^\ast} \frac{2+h\chi_4(y-x)+\overline{h}\overline{\chi_4}(y-x)}{4}.
	\end{align*}
	A graph $G$ is called vertex-transitive if given any two vertices $v_1$ and $v_2$, there exists a graph automorphism $f: G\rightarrow G$ such that $f(v_1)=v_2$.
	\begin{proposition}
		Let $n=p^\alpha$ or $2p^\alpha$, where $p\equiv 1\pmod 8$ and $\alpha\geq 1$, and let $G^\ast(n)$ be the Peisert-like graph of order $n$. Then, $G^\ast(n)$ is vertex-transitive.
	\end{proposition}
	\begin{proof}
		$G^\ast(n)$ being a Cayley graph, is vertex-transitive; see \cite[Theorem 3.1.2]{god}. We have the following explicit automorphism to demonstrate the same. Let $a\in\mathbb{Z}_n$. Then, the map
		\begin{align*}
			V(G^\ast(n))&\rightarrow V(G^\ast(n))\\
			x&\mapsto x+a
		\end{align*}
		is an automorphism. 
	\end{proof}
	We note here that unlike the Peisert graph, the Peisert-like graph is not self-complementary unless the number of vertices in the graph is a prime. This is because a self complementary graph on $n$ vertices must necessarily have $\frac{n(n-1)}{4}$ edges, but for $n=p^\alpha$ or $n=2p^\alpha$, $\phi(n)\neq n-1$ unless $n$ is a prime.
	We also observe that the Peisert-like graph, although never a cycle graph, has a spanning cycle. So, it is a connected graph. This is because, for each vertex $x\in\mathbb{Z}_n$, the vertices $x+1$ and $x-1$ are both adjacent to $x$.
	\par
	Next, we compute the number of triangles in the graph $G^\ast(n)$. For this purpose we take $n=p^\alpha$ ($p\equiv 1\pmod 8$ being a prime) only, since for the case $n=2p^\alpha$ there are no cliques of order greater than $2$. We first prove the following lemma.
	\begin{lemma}\label{lemz}
		Let $n=p^\alpha$, where $p\equiv 1\pmod 8$ is a prime and $\alpha$ is a positive integer. Let $\chi_4$ be a character on $\mathbb{Z}_n^\ast$ of order $4$. Then, $\chi_4$ has period $p$.
	\end{lemma}
	\begin{proof}
		The proof goes along similar lines as in Lemmas 2.6 and 2.7 in \cite{bhowmik2020paley}. Let $\mathbb{Z}_n^\ast=\langle g\rangle $ and let $x\in\mathbb{Z}_n$. The result holds if $p\mid x$, so let us assume that $x\in\mathbb{Z}_n^\ast$.
		Let $x^{-1}$ denote the multiplicative inverse of $x$ in $\mathbb{Z}_n^\ast$. Then by the binomial theorem, 
		\begin{align}\label{bina}
			(1+px^{-1})^\frac{\phi(n)}{4}=\sum_{i=0}^{\frac{\phi(n)}{4}}{\frac{\phi(n)}{4}\choose i}(p x^{-1})^i.
		\end{align}
		Now, we show that 
		\begin{align}\label{key}
			p^\alpha\mid {\frac{\phi(n)}{4}\choose i}(p x^{-1})^i\text{ for } i=1,\ldots,\frac{\phi(n)}{4}.
		\end{align}
		For $\alpha\leq i\leq \frac{\phi(n)}{4}$, \eqref{key} is evident. So, we assume that $1\leq i\leq \alpha-1$.
		To this end, we observe that
		\begin{align*}
			{\frac{\phi(n)}{4}\choose i}=\dfrac{\frac{\phi(n)}{4}\left( \frac{\phi(n)}{4}-1\right)\cdots \left( \frac{\phi(n)}{4}-i+1\right) }{i!}
		\end{align*}
		where $\frac{\phi(n)}{4}=p^{\alpha-i}p^{i-1}\left( \frac{p-1}{4}\right) $, therefore to show \eqref{key} it is sufficient to show that $p^i$ does not divide $i!$. Let $v_p(i!)$ be the highest power of $p$ dividing $i!$, and let $\sigma_p(i)$ be the sum of digits of the base-$p$ representation of $i$. By Legendre's formula, $v_p(i!)=\sum\limits_{k=1}^{\infty}\lfloor\frac{i}{p^k}\rfloor$, from which it can be deduced that $v_p(i!)=\dfrac{i-\sigma_p(i)}{p-1}$. If $p^i$ divides $i!$ then $v_p(i)\geq i$, that is, $\dfrac{i-v_p(i)}{p-1}\geq i$, which is not possible. This proves \eqref{key}. Thus, \eqref{bina} yields $(1+px^{-1})^{\frac{\phi(n)}{4}}\equiv 1\pmod {p^\alpha}$. So, if $1+p x^{-1}=g^t$ in $\mathbb{Z}_n^\ast$ for some $t\in\mathbb{Z}$, then $1=(1+px^{-1})^\frac{\phi(n)}{4}=g^{\frac{t\phi(n)}{4}}$, which implies that $\phi(n)\mid\frac{t\phi(n)}{4}$, which gives $4\mid t$ and hence, $1+px^{-1}\in \langle g^4\rangle $. This means that $\chi_4(1+p x^{-1})=1$, that is, $\chi_4(x+p)=\chi_4(x)$, completing the proof of the lemma.
	\end{proof}
	Now, we prove Theorem \ref{thm1}.
	\begin{proof}[Proof of Theorem \ref{thm1}] Let $k_3(G^\ast(p^\alpha),0)$ denote the number of triangles in $G^\ast(p^\alpha)$ containing the vertex $0$. Since $G^\ast(p^\alpha)$ is vertex-transitive,  so
		\begin{align}\label{000}
			k_3(G^\ast(p^\alpha))=\frac{p^\alpha}{3}\times k_3(G^\ast(p^\alpha),0).
		\end{align}
		Recall that $\mathbb{Z}_{p^\alpha}^\ast=\langle g\rangle$ and $h=1-\chi_4(g)$. Now, using \eqref{war} we have
		\begin{align}\label{00}
			k_3(G^\ast(p^\alpha),0)=\frac{1}{2}\sum\limits_{x\in\mathbb{Z}_{p^\alpha}^\ast}\sum\limits_{y,x-y\in \mathbb{Z}_{p^\alpha}^\ast}&\left[\frac{2+h\chi_4(x)+\overline{h}\overline{\chi_4}(x)}{4} \times \frac{2+h\chi_4(y)+\overline{h}\overline{\chi_4}(y)}{4}\right.\notag\\ &\left.\times\frac{2+h\chi_4(x-y)+\overline{h}\overline{\chi_4}(x-y)}{4}\right].
		\end{align}
		We shall use the fact that $\chi_4(-1)=1$ since $-1\in\langle g^4\rangle$. Firstly, we evaluate the sum in \eqref{00} indexed by $y$. We have
		\begin{align}\label{0}
			&\sum\limits_{y,x-y\in \mathbb{Z}_{p^\alpha}^\ast}[2+h\chi_4(y)+\overline{h}\overline{\chi_4}(y)][2+h\chi_4(x-y)+\overline{h}\overline{\chi_4}(x-y)] \notag \\
			&=\sum\limits_{y,x-y\in \mathbb{Z}_{p^\alpha}^\ast} [4+2h\chi_4(y)+2\overline{h}\overline{\chi_4}(y)+2h\chi_4(x-y)+2\overline{h}\overline{\chi_4}(x-y)+h^2\chi_4(y(x-y))\notag \\&\hspace{1.5cm}+|h|^2\chi_4(y)\overline{\chi_4}(x-y)
			+|h|^2\overline{\chi_4}(y)\chi_4(x-y)+\overline{h}^2 \overline{\chi_4}(y(x-y))].
		\end{align}
		Using Lemma \ref{lemz}, we find that 
		\begin{align}\label{1}
			\sum\limits_{y,x-y\in \mathbb{Z}_{p^\alpha}^\ast}\chi_4(x-y)&=\sum\limits_{x-y\in \mathbb{Z}_{p^\alpha}^\ast}\chi_4(x-y)-\sum\limits_{\substack{x-y\in\mathbb{Z}_{p^\alpha}^\ast\\p\mid y}}\chi_4(x-y)\notag \\
			&=-\sum_{t=0}^{p^{\alpha-1}-1}\chi_4(x-pt)
			=-\sum_{t=0}^{p^{\alpha-1}-1}\chi_4(x)=-p^{\alpha-1}\chi_4(x),
		\end{align}
		and similarly
		\begin{align}\label{3}
			\sum\limits_{y,x-y\in \mathbb{Z}_{p^\alpha}^\ast}\chi_4(y)=-p^{\alpha-1}\chi_4(x).
		\end{align}
		Using the substitution $y\mapsto xy$ in the following sum, we have
		\begin{align}\label{4}
			\sum\limits_{y,x-y\in \mathbb{Z}_{p^\alpha}^\ast}\chi_4(y(x-y))=\sum\limits_{y\in\mathbb{Z}_{p^\alpha}} \chi_4(y(x-y))=\varphi(x) J(\chi_4,\chi_4).
		\end{align}
		Also, we find that
		\begin{align}\label{5}
			&	\sum\limits_{y,x-y\in \mathbb{Z}_{p^\alpha}^\ast} \chi_4(y)\overline{\chi_4}(x-y)=\sum\limits_{p\nmid y}\overline{\chi_4}\left( xy^{-1}-1\right),
		\end{align}
		where $y^{-1}$ denotes the multipicative inverse of $y$ in $\mathbb{Z}_n^\ast$.
		The following map
		\begin{align*}
			\{y\in\mathbb{Z}_{p^\alpha}:p\nmid y,x-y\}&\rightarrow \{z\in\mathbb{Z}_{p^\alpha}:p\nmid z, z+1\}\\
			y&\mapsto xy^{-1}-1
		\end{align*}
		is a bijection, and hence, $\eqref{5}$ yields
		\begin{align}\label{6}
			\sum\limits_{y,x-y\in \mathbb{Z}_{p^\alpha}^\ast} \chi_4(y)\overline{\chi_4}(x-y)=\sum\limits_{p\nmid z+1}\overline{\chi_4}(z)=-\sum\limits_{p\mid z+1}\overline{\chi_4}(z)=-p^{\alpha-1}.
		\end{align}
		Lastly, we have 
		\begin{align}\label{7}
			\sum\limits_{y,x-y\in \mathbb{Z}_{p^\alpha}^\ast}1=\sum\limits_{p\nmid y}1-\sum\limits_{\substack{p\nmid y\\p\mid x-y}}1=p^{\alpha-1}(p-2).
		\end{align}
		Employing \eqref{1} - \eqref{7} in $\eqref{0}$, and then combining with $\eqref{00}$ we find that
		\begin{align}\label{numb}
			k_3(G^\ast(p^\alpha),0)=\frac{1}{128}\sum\limits_{p\nmid x}[2+h\chi_4(x)+\overline{h}\overline{\chi_4}(x)] [A-B\chi_4(x)-\overline{B}\overline{\chi_4}(x)+C\varphi(x)],
		\end{align}
		where
		\begin{align*}
			A&=4(p-3)p^{\alpha-1},\\
			B&=4hp^{\alpha-1},\text{ and}\\
			C&=h^2 J(\chi_4,\chi_4)+\overline{h}^2 \cdot \overline{J(\chi_4,\chi_4)}.
		\end{align*}
		After expanding the expression inside the sum over $x$ and proceeding similarly as shown above, \eqref{numb} yields 
		\begin{align}\label{8}
			k_3(G^\ast(p^\alpha),0)	&=\frac{1}{128}[2A-\overline{B}h-B\overline{h}]\phi(p^\alpha)\notag \\
			&=\frac{1}{16}p^{2\alpha-2}(p-1)(p-5).
		\end{align}
		Finally, combining $\eqref{8}$ and $\eqref{000}$, we obtain the required result.
	\end{proof}
	\section{Hypergeometric functions for Dirichlet characters}\label{sechyp}
	In this section, we introduce hypergeometric functions having Dirichlet characters modulo $p^\alpha$ as arguments, where $p$ is an odd prime and $\alpha$ is a positive integer. Firstly, we study some character sums involving Dirichlet characters. 
	Let $s,t\in\mathbb{Z}_{p^\alpha}$. We define the function $\delta_s(t)$ as
	\begin{align}
		\delta_s(t)=\left\{ \begin{array}{cl} 1, &\text{if } s=t;\\
			0, & \text{otherwise. }\end{array}\right. 	
	\end{align}
	\begin{lemma}\label{le1}
		Let $q=p^\alpha$, where $p$ is an odd prime and $\alpha\geq 1$ is an integer. Let $A$ be a Dirichlet character mod $q$. For $x\in\mathbb{Z}_q$, we have
		\begin{align}\label{eqn-new-01}
			A(1+x)=\sum_{t=0}^{p^{\alpha-1}-1}A(1+tp)\delta_{tp}(x)+\frac{1}{\phi(q)}\sum_{\chi\in\widehat{\mathbb{Z}_q^\ast}} J(A,\overline{\chi})\chi(-x).		
		\end{align}
	\end{lemma}
	\begin{proof}
		For $a\in\mathbb{Z}_q^\ast$, we have
		\begin{align*}
			\frac{1}{\phi(q)}\sum_{\chi\in\widehat{\mathbb{Z}_q^\ast}} \chi(x)\overline{\chi}(a) = \left\{ \begin{array}{cl} 1, &\text{if } x=a;\\
				0, & \text{otherwise. }\end{array}\right. 	
		\end{align*}
		Hence, we have
		\begin{align*}
			A(1+x)&=\sum_{t=0}^{p^{\alpha-1}-1}A(1+tp)\delta_{tp}(x)+\sum_{a\in\mathbb{Z}_q^\ast}A(1+a)\delta_a(x)\\
			&=\sum_{t=0}^{p^{\alpha-1}-1}A(1+tp)\delta_{tp}(x)+\frac{1}{\phi(q)}\sum_{\chi\in\widehat{\mathbb{Z}_q^\ast}} \chi(x)\sum_{a\in\mathbb{Z}_q^\ast}A(1+a)\overline{\chi}(a).
		\end{align*}
		It is easy to see that
		\begin{align*}
			\sum_{a\in\mathbb{Z}_q}A(1+a)\overline{\chi}(a)=	\sum_{a\in\mathbb{Z}_q}A(1-a)\overline{\chi}(-a)=\chi(-1)J(A,\overline{\chi}),
		\end{align*}
		which completes the proof of the lemma.
	\end{proof}
	Greene observed in \cite{greene} that the finite field analogue of the binomial coefficient is the Jacobi sum. Following Greene, we define binomial coefficient for Dirichlet characters. 
	\begin{definition}
		Let $q=p^\alpha$, where $p$ is an odd prime and $\alpha\geq 1$ is an integer. For Dirichlet characters $A$ and $B$ mod $q$, we define $\binom{A}{B}:=\frac{B(-1)}{q}J(A,\overline{B})$.
	\end{definition}
	We can rewrite \eqref{eqn-new-01} in terms of binomial coefficients as follows.
	\begin{align}\label{bin}
		A(1+x)=\sum_{t=0}^{p^{\alpha-1}-1}A(1+tp)\delta_{tp}(x)+\frac{q}{\phi(q)}\sum_{\chi\in\widehat{\mathbb{Z}_q^\ast}} \binom{A}{\chi}\chi(x).
	\end{align}
	In the following lemma, we state some properties of the binomial coefficients.
	\begin{lemma}
		Let $q=p^\alpha$, where $p$ is an odd prime and $\alpha\geq 1$ is an integer. For Dirichlet characters $A$ and $B$ mod $q$, we have 
		\begin{align}
			\label{ab}	\binom{A}{B}&=\binom{A}{A\overline{B}};\\
			\label{abs}\binom{A}{B}&=\binom{\overline{A}B}{B}B(-1);\\
			\label{coeff}\binom{A}{B}&=\binom{\overline{B}}{\overline{A}}AB(-1).
		\end{align}
	\end{lemma}
	\begin{proof}
		We prove \eqref{ab}. By definition, 
		\begin{align}\label{rev1}
			\binom{A}{A\overline{B}}&=\frac{AB(-1)}{q}J(A, \overline{A}B)\notag\\
			&=\frac{AB(-1)}{q}\sum_{p\nmid x,1-x}A(x)\overline{A}B(1-x)\notag\\
			&=\frac{AB(-1)}{q} \sum_{p\nmid x,1-x}A(x(1-x)^{-1})B(1-x).
		\end{align}
		The following map
		\begin{align*}
			\{x\in\mathbb{Z}_q: p\nmid x,1-x\}&\rightarrow \{y\in\mathbb{Z}_q: p\nmid y,y+1\}\\
			x&\mapsto x(1-x)^{-1}
		\end{align*}
		is a bijection, so \eqref{rev1} yields
		\begin{align*}
			\binom{A}{A\overline{B}}&=\frac{AB(-1)}{q}\sum_{p\nmid y,y+1}A(y)\overline{B}(1+y)\\
			&=\frac{AB(-1)}{q}\sum_{p\nmid y,y+1}A(-y)\overline{B}(1-y)\\
			&=\frac{AB(-1)}{q}A(-1)J(A,\overline{B}),
		\end{align*}
		which equals $\binom{A}{B}$. This proves  \eqref{ab}. The proofs of \eqref{abs} and \eqref{coeff} follow in a similar fashion, using the definition of binomial coefficient and the bijection used in the proof of \eqref{ab}.
	\end{proof}
	The following definition can be considered as a $\mathbb{Z}_{p^\alpha}$-analogue for the integral representation of the classical hypergeometric series.
	\begin{definition}\label{def01}
		Let $q=p^\alpha$, where $p$ is an odd prime and $\alpha\geq 1$ is an integer. Let $A,B$ and $C$ be Dirichlet characters mod $q$ and let $\varepsilon$ be the trivial character mod $q$. Then, for $x\in\mathbb{Z}_q$, we define 	
		\begin{align*}
			{_{2}}F_{1}\left( \begin{array}{cc} A, &B \\ &C \end{array} |x \right):=\frac{\varepsilon(x)BC(-1)}{q}\sum_{y\in\mathbb{Z}_q}B(y)\overline{B}C(1-y)\overline{A}(1-xy).
		\end{align*}
	\end{definition}
	In the following lemma we express the hypergeometric function in terms of the binomial coefficients. This is an analogue of Theorem 3.6 in \cite{greene}.
	\begin{lemma}\label{cd}
		Let $q=p^\alpha$, where $p$ is an odd prime and $\alpha\geq 1$ is an integer. For Dirichlet characters $A,B$ and $C$ mod $q$, 
		\begin{align*}
			{_{2}}F_{1}\left( \begin{array}{cc} A, &B \\ &C \end{array} |x \right)=\frac{q}{\phi(q)}\sum_{\chi\in\widehat{\mathbb{Z}_q^\ast}} \binom{A\chi}{\chi}\binom{B\chi}{C\chi}\chi(x).
		\end{align*}
	\end{lemma}
	\begin{proof}
		Let $y\in\mathbb{Z}_q$. By \eqref{bin}, we have
		\begin{align}\label{rev2}
			\overline{A}(1-xy)=\sum_{t=0}^{p^{\alpha-1}-1}\overline{A}(1+tp)\delta_{tp}(-xy)+\frac{q}{\phi(q)}\sum_{\chi\in\widehat{\mathbb{Z}_q^\ast}}\binom{\overline{A}}{\chi}\chi(-xy).
		\end{align}
		Using \eqref{abs}, \eqref{rev2} yields
		\begin{align}\label{rev3}
			\overline{A}(1-xy)=\sum_{t=0}^{p^{\alpha-1}-1}\overline{A}(1+tp)\delta_{tp}(-xy)+\frac{q}{\phi(q)}\sum_{\chi\in\widehat{\mathbb{Z}_q^\ast}}\binom{A\chi}{\chi}\chi(xy).
		\end{align}
		Substituting \eqref{rev3} in Definition \ref{def01} and noting that $\varepsilon(x)B(y)\delta_{tp}(-xy)=0$ for all $x$ and $y$ yields
		\begin{align*}
			{_{2}}F_{1}\left( \begin{array}{cc} A, &B \\ &C \end{array} |x \right)&=\frac{BC(-1)}{\phi(q)}\sum_{y\in\mathbb{Z}_q}\sum_{\chi\in\widehat{\mathbb{Z}_q^\ast}}\binom{A\chi}{\chi}	\chi(x) B\chi(y)\overline{B}C(1-y)\notag\\
			&=\frac{BC(-1)}{\phi(q)}\sum_{\chi\in\widehat{\mathbb{Z}_q^\ast}}\binom{A\chi}{\chi}J(B\chi,\overline{B}C)\chi(x)\notag\\
			&=\frac{q}{\phi(q)}\sum_{\chi\in\widehat{\mathbb{Z}_q^\ast}}\binom{A\chi}{\chi}\binom{B\chi}{B\overline{C}}\chi(x),
		\end{align*}
		and we complete the proof by using \eqref{ab}.
	\end{proof}
	We now define hypergeometric functions containing Dirichlet characters for any $n\geq 1$.
	\begin{definition}\label{nfn}
		Let $q=p^\alpha$, where $p$ is an odd prime and $\alpha\geq 1$ is an integer. For Dirichlet characters $A_0,$ $A_1,$ $\ldots,$ $A_n,$ and $B_1, \ldots, B_n$ mod $q$ and $x \in \mathbb{Z}_q$, the hypergeometric function $_{n+1}F_n$ is defined by $${_{n+1}} F_n\left(\begin{matrix}A_0,&A_1,& \ldots, &A_n&\\&B_1,& \ldots, &B_n&\end{matrix} | x \right):= \frac{q}{\phi(q)} \sum_{\chi\in\widehat{\mathbb{Z}_q^\ast}} {{A_0 \chi} \choose {\chi}} {{A_1 \chi} \choose {B_1 \chi}}\cdots {{A_n \chi} \choose {B_n \chi}} \chi(x).$$
	\end{definition}
	We have the following recursive formula, whose proof follows the same way as in the proof of Theorem 3.13 in \cite{greene}.
	\begin{lemma}\label{lemahyp}
		Let $q=p^\alpha$, where $p$ is an odd prime and $\alpha\geq 1$ is an integer. For Dirichlet characters $A_0,$ $A_1,$ $\ldots,$ $A_n,$ and $B_1, \ldots, B_n$ mod $q$ and $x \in \mathbb{Z}_q$, we have
		\begin{align*}
			&	{_{n+1}} F_n\left(\begin{matrix}A_0,&A_1,& \ldots, &A_n&\\&B_1,& \ldots, &B_n&\end{matrix}|x \right)\\
			&= \frac{A_nB_n(-1)}{q} \sum_{y} {_{n}} F_{n-1}\left(\begin{matrix}A_0,&A_1,& \ldots, &A_{n-1}&\\&B_1,& \ldots, &B_{n-1}&\end{matrix} | xy \right)A_n(y)\overline{A_n}B_n(1-y).
		\end{align*}
	\end{lemma}
	\begin{proof}
		Let $\chi\in\widehat{\mathbb{Z}_q^\ast}$.	Using \eqref{ab}, we find that
		\begin{align}\label{rev5}
			\binom{A_n\chi}{B_n\chi}&=\binom{A_n\chi}{A_n\overline{B_n}}\notag\\
			&=\frac{A_n B_n(-1)}{q}J(A_n\chi,\overline{A_n}B_n)\notag\\
			&=\frac{A_n B_n(-1)}{q}\sum_{y\in\mathbb{Z}_q}A_n\chi(y)\overline{A_n}B_n(1-y).
		\end{align}
		Then, using \eqref{rev5} in Definition \ref{nfn}, we have
		\begin{align*}
			&{_{n+1}} F_n\left(\begin{matrix}A_0,&A_1,& \ldots, &A_n&\\&B_1,& \ldots, &B_n&\end{matrix} | x \right)\\
			&=\frac{q}{\phi(q)}\hspace{-0.1cm} \sum_{\chi\in\widehat{\mathbb{Z}_q^\ast}}\hspace{-0.1cm}\left[\hspace{-0.1cm} {{A_0 \chi} \choose {\chi}}\hspace{-0.1cm} {{A_1 \chi} \choose {B_1 \chi}}\hspace{-0.1cm}\cdots\hspace{-0.1cm} {{A_{n-1} \chi} \choose {B_{n-1} \chi}} \chi(x)\hspace{-0.1cm}
			\left(\frac{A_n B_n(-1)}{q}\hspace{-0.1cm}\sum_{y\in\mathbb{Z}_q}A_n\chi(y)\overline{A_n}B_n(1-y)\right)\hspace{-0.1cm}\right]\\
			&=\frac{A_n B_n(-1)}{q}\hspace{-0.1cm}\sum_{y\in\mathbb{Z}_q}\hspace{-0.1cm}\left(\frac{q}{\phi(q)}\sum_{\chi\in\widehat{\mathbb{Z}_q^\ast}}\hspace{-0.1cm}{{A_0 \chi} \choose {\chi}}\hspace{-0.1cm} {{A_1 \chi} \choose {B_1 \chi}}\hspace{-0.1cm}\cdots\hspace{-0.1cm} {{A_{n-1} \chi} \choose {B_{n-1} \chi}} \chi(xy)\right)\hspace{-0.1cm}A_n(y)\overline{A_n}B_n(1-y),
		\end{align*}
		and we complete the proof by noting Definition \ref{nfn} again.
	\end{proof}
	We have the following corollary, which is an analogue of Corollary 3.14 in \cite{greene}.
	\begin{corollary}
		Let $q=p^\alpha$, where $p$ is an odd prime and $\alpha\geq 1$ is an integer. Let $A,B,C,D$ and $E$ be Dirichlet characters mod $q$ and let $\varepsilon$ be the trivial character mod $q$. Then,
		\begin{align}
			\label{rev6}{_{3}}F_{2} \left( \begin{array}{ccc} A,&B,&C\\&D,&E \end{array} | x \right)
			&=\frac{\varepsilon(x)BCDE(-1)}{q^2}\notag\\
			&\hspace{.2cm}\times \sum_{y,z}C(y)\overline{C}E(1-y)B(z)	\overline{B}D(1-z)\overline{A}(1-xyz), \\
			\label{corr2}	{_{3}}F_{2} \left( \begin{array}{ccc} A,&B,&C\\&D,&E \end{array} | x \right)
			&=\frac{\varepsilon(x)BD(-1)}{q^2}\notag \\
			&\hspace{.2cm}\times \sum_{y,z}A\overline{E}(y)\overline{C}E(1-y)B(z)	\overline{B}D(1-z)\overline{A}(y-xz).
		\end{align}
	\end{corollary}
	\begin{proof}
		The proof of \eqref{rev6} follows from Lemma \ref{lemahyp} and Definition \ref{def01}. To prove \eqref{corr2}, we note that
		\begin{align*}
			\{x\in\mathbb{Z}_q: p\nmid y,1-y\}&\rightarrow \{y'\in\mathbb{Z}_q: p\nmid y',1-y'\}\\
			y&\mapsto y^{-1}
		\end{align*}
		is a bijection. So, we use the substitution $y'=y^{-1}$ in the sum indexed by $y$ in \eqref{rev6} and readily obtain \eqref{corr2}.
		
	\end{proof}
	\section{some lemmas required to prove Theorem \ref{thm2}}
	In this section, we evaluate some character sums which we come across in the proof of Theorem \ref{thm2}. We also prove some relations between hypergeometric functions as in \cite{greene}. We note that if $p\equiv 1\pmod 8$ is a prime and $\alpha$ is a positive integer, then $\chi_4(-1)=1$, where $\chi_4$ is a Dirichlet character mod $p^\alpha$ of order $4$. The following three lemmas are analogues of Lemmas 2.2 to 2.6 in \cite{bhowmikpeisert}.
	\begin{lemma}\label{lemsec1}
		For a prime $p\equiv 1\pmod 8$ and an integer $\alpha\geq 1$, let $\chi_4$ be a Dirichlet character mod $p^\alpha$ of order $4$ and let $\varphi$ be the quadratic character mod $p^\alpha$. 
		Let $x\in \mathbb{Z}_{p^\alpha}^\ast$ be such that $p\nmid 1-x$. Let $\rho:=J(\chi_4,\chi_4)$. Then, we have 
		\begin{align*}
			& \sum\limits_{\substack{y\in \mathbb{Z}_{p^\alpha}^\ast\\p\nmid 1-y,x-y}}\chi_4^{i_1}(x-y)\chi_4^{i_2}(1-y)\chi_4^{i_3}(y) \\
			&= \left\{
			\begin{array}{lll}
				p^{\alpha-1}(p-3), & \hbox{if $(i_1, i_2, i_3)=(0, 0, 0)$}; \\
				-p^{\alpha-1}(1+\chi_4(x)), & \hbox{if $(i_1, i_2, i_3)=(0, 0, 1)$};\\
				-p^{\alpha-1}(1+\chi_4(1-x)), & \hbox{if $(i_1, i_2, i_3)=(0, 1, 0)$};\\
				-p^{\alpha-1}(\chi_4(1-x)+\chi_4(x)),& \hbox{if $(i_1, i_2, i_3)=(1, 0, 0)$};\\
				\rho-p^{\alpha-1}\chi_4(1-x)\chi_4(x),& \hbox{if $(i_1, i_2, i_3)=(0, 1, 1)$};\\
				\varphi(x)\rho-p^{\alpha-1}\chi_4(1-x), & \hbox{if $(i_1, i_2, i_3)=(1, 0, 1)$};\\
				\varphi(x-1)\rho-p^{\alpha-1}\chi_4(x), & \hbox{if $(i_1, i_2, i_3)=(1, 1, 0)$};\\
				-p^{\alpha-1}(1+\overline{\chi_4}(1-x)\chi_4(x)), & \hbox{if $(i_1, i_2, i_3)=(0, -1, 1)$};\\
				-p^{\alpha-1}(1+\overline{\chi_4}(1-x)), & \hbox{if $(i_1, i_2, i_3)=(-1, 0, 1)$};\\
				-p^{\alpha-1}(1+\overline{\chi_4}(x)), & \hbox{if $(i_1, i_2, i_3)=(-1, 1, 0)$}.
			\end{array}
			\right.
		\end{align*}		 
	\end{lemma}
	\begin{proof}
		The proofs are straightforward, and we give one such instance. Let $(i_1, i_2, i_3)=(0, -1, 1)$. Since $\chi_4$ is of period $p$, we have 
		\begin{align}\label{ef}
			\sum\limits_{\substack{y\in \mathbb{Z}_{p^\alpha}^\ast\\p\nmid 1-y,x-y}} \chi_4(y)\overline{\chi_4}(1-y)=\sum_{p\nmid y,1-y} \chi_4(y)\overline{\chi_4}(1-y)-p^{\alpha-1}\chi_4(x)\overline{\chi_4}(1-x).
		\end{align}
		Now, the following map
		\begin{align*}
			&\{y\in\mathbb{Z}_{p^\alpha}: p\nmid y,y-1\}\rightarrow \{z\in\mathbb{Z}_{p^\alpha}: p\nmid z,1+z\}\\
			&\hspace{3cm}y\mapsto y(1-y)^{-1}
		\end{align*}
		is a bijection. Hence, 
		\begin{align}\label{new-eqn-02}
			\sum_{p\nmid y,1-y} \chi_4(y)\overline{\chi_4}(1-y)=-p^{\alpha-1}.
		\end{align}
		Combining \eqref{ef} and \eqref{new-eqn-02}, we complete the proof of the lemma when $(i_1, i_2, i_3)=(0, -1, 1)$.
	\end{proof}
	The proofs of the following two lemmas are similar to that of Lemma \ref{lemsec1} and involve the same techniques, so we state them without proofs.
	\begin{lemma}\label{lema1}
		Let $p\equiv 1\pmod 8$ be a prime and let $\alpha\geq 1$ be an integer. Let $\chi_4$ be a Dirichlet character mod $p^\alpha$ of order $4$ and let $\varphi$ be the quadratic character mod $p^\alpha$. Let $\xi:=J(\chi_4,\varphi)$. Then, we have 
		\begin{align*}
			&\sum\limits_{p\nmid x,1-x}\hspace{.2cm}\sum\limits_{p\nmid y,1-y,x-y} \chi_4^{i_1}(y)\chi_4^{i_2}(1-y)\chi_4^{i_3}(x-y)\\
			&=\left\{
			\begin{array}{lll}
				-2\xi p^{\alpha-1}, & \hbox{if $(i_1, i_2, i_3)=(1, 1, 1);$} \\
				2p^{2\alpha-2}, & \hbox{if $(i_1, i_2, i_3)= (1, 1, -1);$} \\
				-p^{\alpha-1}(\overline{\xi}-p^{\alpha-1}), & \hbox{if $(i_1, i_2, i_3)=(1, -1, 1);$}\\
				-p^{\alpha-1}(\xi-p^{\alpha-1}), & \hbox{if $(i_1, i_2, i_3)=(1, -1, -1).$}
			\end{array}
			\right.
		\end{align*}	
	\end{lemma}
	\begin{lemma}\label{corr}
		Let $p\equiv 1\pmod 8$ be a prime and let $\alpha\geq 1$ be an integer. Let $\chi_4$ be a Dirichlet character mod $p^\alpha$ of order $4$ and let $\varphi$ be the quadratic character mod $p^\alpha$. Let $\rho:=J(\chi_4,\chi_4)$,  $\xi:=J(\chi_4,\varphi)$, $S_1:=-p^{\alpha-1}(\rho+\xi)$, $S_2:=-p^{\alpha-1}\rho+p^{2\alpha-2}$, $S_3:=|\rho|^2+p^{2\alpha-2},S_4:=p^{2\alpha-2}-p^{\alpha-1}\xi,S_5:=\rho^2-p^{\alpha-1}\overline{\xi}$ and $S_6:=2p^{2\alpha-2}$. Then, for  $i_1,i_2,i_3\in\{\pm 1\}$, we have the following tabulation of the values of the expression given below:
		\begin{align}\label{new-eqn1}
			\sum\limits_{\substack{x,y\in\mathbb{Z}_{p^\alpha},\\ p\nmid x, 1-x}}A_x \cdot \chi_4^{i_1}(y)\chi_4^{i_2}(1-y)\chi_4^{i_3}(x-y).
		\end{align}
		For $w\in\{1,\ldots,8\}$ and $z\in \{1,\ldots,7\}$, the $(w,z)$-th entry in the table corresponds to \eqref{new-eqn1},
		where $A_x$ is either $\chi_4(x),\overline{\chi_4}(x),\chi_4(1-x)$ or $\overline{\chi_4}(1-x)$ and the tuple $(i_1,i_2,i_3)$ depends on $w$.
		\begin{align*}
			\begin{array}{|l|l|l|l|l|l|l|}
				\cline {4 - 7 } \multicolumn{3}{c|}{} & \multicolumn{4}{|c|}{A_{x}} \\
				\hline i_{1} & i_{2} & i_{3} & \chi_4(x) & \overline{\chi_4}(x) & \chi_4(1-x) & \overline{\chi_4}(1-x) \\
				\hline
				1 & 1 & 1  & S_1 & S_1 & S_1 & S_1 \\
				1 & 1 & -1  & S_2 & \overline{S_2} & S_2 & \overline{S_2} \\
				1 & -1 & 1 & S_3 & S_6 & S_5 & \overline{S_4} \\
				1 & -1 & -1 & S_4 & \overline{S_5} & S_6 & S_3 \\
				-1 & 1 & 1 & S_5  & \overline{S_4} & S_3  & S_6 \\
				-1 & 1 & -1 & S_6  & S_3 & S_4  & \overline{S_5} \\
				-1 & -1 & 1 & S_2 & \overline{S_2} & S_2 & \overline{S_2} \\
				-1 & -1 & -1 & \overline{S_1}  & \overline{S_1} &\overline{S_1}  & \overline{S_1}\\
				\hline 
			\end{array}	
		\end{align*}
		For example, the $(3,6)$-th position contains the value $S_5=\rho^2-p^{\alpha-1}\overline{\xi}$. Here $w=3$ corresponds to $i_1=1,i_2=-1,i_3=1$; $z=6$ corresponds to the column $A_x=\chi_4(1-x)$. 
	\end{lemma}
	Now, we shall observe that equations $(2.9)$ to $(2.15)$ in \cite{bhowmikpeisert} also hold if we replace multiplicative characters on a finite field by Dirichlet characters mod $p^\alpha$. In \cite{bhowmikpeisert}, we used Lemma 2.8 therein, in the proof of finding cliques of order 4 in the Peisert graph; here we shall follow a similar approach. Recalling Definition \ref{nfn},  we have
	\begin{align}\label{3f21}
		{_{3}}F_{2} \left( \begin{array}{ccc} A,&B,&C\\&D,&E \end{array} | 1 \right) = \frac{p^\alpha}{\phi(p^\alpha)} \sum_{\chi\in\widehat{\mathbb{Z}_{p^\alpha}^\ast}} \binom{A\chi}{\chi} \binom{B\chi}{D\chi} \binom{C\chi}{E\chi} \chi(1).
	\end{align}
	Below are three lemmas whose proofs involve change of variable in the sum in \eqref{3f21}.
	The following lemma is a $\mathbb{Z}_{p^\alpha}$-analogue of Theorem 4.2 (i) in \cite{greene}.
	\begin{lemma}\label{lem001}
		Let $p$ be an odd prime and let $\alpha\geq 1$ be an integer. Let $A,B,C,D,E$ be Dirichlet characters mod $p^\alpha$. Then,
		\begin{align*}
			{_{3}}F_{2} \left( \begin{array}{ccc} A,&B,&C\\&D,&E \end{array} | 1 \right) = {_{3}}F_{2}\left( \begin{array}{ccc} B\overline{D}, &A\overline{D}, &C\overline{D}\\ &\overline{D}, &E\overline{D} \end{array} |1 \right) .
		\end{align*}
	\end{lemma}
	\begin{proof}
		Employing the transformation $\chi \mapsto \overline{D}\chi$ in \eqref{3f21} yields the required result.	
	\end{proof}
	The following lemma is a $\mathbb{Z}_{p^\alpha}$-analogue of Theorem 4.2 (ii) in \cite{greene}.
	\begin{lemma}\label{lem002}
		Let $p$ be an odd prime and let $\alpha\geq 1$ be an integer. Let $A,B,C,D,E$ be Dirichlet characters mod $p^\alpha$. Then,
		\begin{align*}
			{_{3}}F_{2}\left( \begin{array}{ccc} A,&B,&C\\&D,&E \end{array} | 1 \right)= ABCDE(-1) {_{3}}F_{2}\left(\begin{array}{ccc} A, &A\overline{D}, &A\overline{E} \\ & A\overline{B}, &A\overline{C} \end{array} | 1 \right).	
		\end{align*}
	\end{lemma}
	\begin{proof} 
		We employ the transformation $\chi \mapsto \overline{A\chi}$ in \eqref{3f21}, and then use \eqref{coeff} to complete the proof.
	\end{proof}
	\begin{lemma}\label{lem003}
		Let $p$ be an odd prime and let $\alpha\geq 1$ be an integer. Let $A,B,C,D,E$ be Dirichlet characters mod $p^\alpha$. Then,
		\begin{align*}
			{_{3}}F_{2}\left( \begin{array}{ccc} A,&B,&C\\&D,&E \end{array} | 1 \right)= ABCDE(-1) {_{3}}F_{2}\left(\begin{array}{ccc} B\overline{D}, &B, &B\overline{E}\\ &B\overline{A}, &B\overline{C} \end{array} | 1 \right).
		\end{align*}		
	\end{lemma}	
	\begin{proof}
		Employing the transformation $\chi\mapsto \overline{B\chi}$ in \eqref{3f21}, and then using \eqref{coeff} we complete the proof.
	\end{proof}
	We further prove $\mathbb{Z}_{p^\alpha}$-analogues of certain transformations satisfied by the Greene's finite field hypergeometric functions. We shall evoke Definition \ref{def01} and Lemma \ref{lemahyp} multiple times.
	Following is an $\mathbb{Z}_{p^\alpha}$-analogue of (4.23) in \cite{greene}.	
	\begin{lemma}\label{lem004}
		Let $p$ be an odd prime and let $\alpha\geq 1$ be an integer. Let $A,B,C,D,E$ be Dirichlet characters mod $p^\alpha$. Then,
		\begin{align*}
			{_{3}}F_{2} \left( \begin{array}{ccc} A,&B,&C\\&D,&E \end{array} | 1 \right)= AE(-1) {_{3}}F_{2}\left(\begin{array}{ccc} A, &B, &\overline{C}E \\ &AB\overline{D}, &E \end{array} | 1 \right) .	
		\end{align*}
	\end{lemma}
	\begin{proof}
		We first show that for $x\in\mathbb{Z}_{p^\alpha}^\ast$, if $p\nmid 1-x$, then
		\begin{align}\label{artf}	
			{_{2}}F_{1}\left( \begin{array}{cc} A, &B \\ &D \end{array} |x \right) &= A(-1) {_{2}}F_{1}\left( \begin{array}{cc} A, &B\\ &AB\overline{D} \end{array} | 1-x \right).
		\end{align}
		To prove \eqref{artf}, let $x\in\mathbb{Z}_{p^\alpha}^\ast$ be such that $p\nmid 1-x$. By Definition \ref{def01}, we have
		\begin{align}\label{artz}
			{_{2}}F_{1}\left( \begin{array}{cc} A, &B\\ &D \end{array} |x \right) = \frac{1\times BD(-1)}{p^\alpha} \sum_{p \nmid y, 1-y,1-xy} B(y) \overline{B}D(1-y)\overline{A}(1-xy).	
		\end{align}
		We find that
		\begin{align*} \{y\in\mathbb{Z}_{p^\alpha}:p \nmid y, 1-y,1-xy\} &\rightarrow \{z\in\mathbb{Z}_{p^\alpha}:p \nmid z,1-z,1-(1-x)z\}\\
			y &\mapsto y(y-1)^{-1}
		\end{align*}
		is a bijection. Hence, \eqref{artz} yields
		\begin{align}\label{artv}
			{_{2}}F_{1}\left( \begin{array}{cc} A, &B\\ &D \end{array} |x \right)&=\frac{BD(-1)}{p^\alpha}\sum_{\substack{p \nmid z,1-z,\\1-(1-x)z}}[B(z(z-1)^{-1})\overline{B}D(-(z-1)^{-1})\notag\\
			&\hspace{3cm}\times\overline{A} ((z-1-xz)(z-1)^{-1})] \notag\\
			&=\frac{D(-1)}{p^\alpha} \sum_{\substack{p \nmid z,1-z,\\1-(1-x)z}}B(z) A\overline{D}(1-z)\overline{A}(1-(1-x)z).
		\end{align}
		Thus, by Definition \ref{def01} and \eqref{artv}, and noting that $\varepsilon(1-x)=1$, we conclude  \eqref{artf}.
		Now, Lemma \ref{lemahyp} and \eqref{artf} give
		\begin{align*}
			{_{3}}F_{2} \left( \begin{array}{ccc} A,&B,&C\\&D,&E \end{array} | 1 \right)&=\frac{CE(-1)}{p^\alpha}\sum_{p\nmid x,1-x} {_{2}}F_{1}\left( \begin{array}{cc} A, &B \\ &D \end{array} |x \right) C(x) \overline{C}E(1-x)\\
			& = \frac{ACE(-1)}{p^\alpha}\sum_{p\nmid x,1-x} {_{2}}F_{1}\left( \begin{array}{cc} A, &B \\ &AB\overline{D}\end{array} |1-x \right)C(x)\overline{C}E(1-x)\\
			& = \frac{ACE(-1)}{p^\alpha}\sum_{p\nmid x,1-x} {_{2}}F_{1}\left( \begin{array}{cc} A, &B \\ &AB\overline{D}\end{array} |x \right)C(1-x)\overline{C}E(x)\\
			&=AE(-1){_{3}}F_{2} \left( \begin{array}{ccc} A,&B,&\overline{C}E\\&AB\overline{D},&E \end{array} | 1 \right),
		\end{align*}
		where we have used the substitution $x\mapsto 1-x$ in the penultimate line. This completes the proof of the lemma.
	\end{proof}
	The following lemma gives a $\mathbb{Z}_{p^\alpha}$-analogue of (4.24) in \cite{greene}.
	\begin{lemma}\label{lem005}
		Let $p$ be an odd prime and let $\alpha\geq 1$ be an integer. Let $A,B,C,D,E$ be Dirichlet characters mod $p^\alpha$. Then,
		$${_{3}}F_{2}\left(\begin{array}{ccc}A, & B, & C\\ & D, & E\end{array}| 1\right)=AD(-1){_{3}}F_{2}\left(\begin{array}{ccc}A, & \overline{B}D, & C\\ & D, & AC\overline{E}\end{array}| 1\right).$$	
	\end{lemma}
	\begin{proof}
		Putting $x=1$ in \eqref{corr2} and using the substitutions $y'=1-y$ and $z'=1-z$ in the double summation therein yield the required result.
	\end{proof}
	The following lemma gives a $\mathbb{Z}_{p^\alpha}$-analogue of (4.25) in \cite{greene}.
	\begin{lemma}\label{lem006}
		Let $p$ be an odd prime and $\alpha\geq 1$ be an integer. Let $A,B,C,D,E$ be Dirichlet characters mod $p^\alpha$. Then,
		$${_{3}}F_{2}\left(\begin{array}{ccc}A, & B, & C\\ & D, & E\end{array}| 1\right)=B(-1){_{3}}F_{2}\left(\begin{array}{ccc}\overline{A}D, & B, & C\\ & D, & BC\overline{E}\end{array}| 1\right).$$	
	\end{lemma}
	\begin{proof}
		At first, we show that if $x\in\mathbb{Z}_{p^\alpha}$ such that $p\nmid 1-x$, then
		\begin{align}\label{arty}
			{_{2}}F_{1}\left(\begin{array}{cc}A, & B\\ & D\end{array}|x\right)=
			\overline{B}(1-x) {_{2}}F_{1}\left(\begin{array}{cc}\overline{A}D, & B\\ & D\end{array}|x(x-1)^{-1}\right).
		\end{align}
		To prove this, let $x\in\mathbb{Z}_{p^\alpha}$ be such that $p\nmid 1-x$. We begin by employing Definition \ref{def01} to obtain
		\begin{align}\label{6.9}
			{_{2}}F_{1}\left(\begin{array}{cc}A, & B\\ & D\end{array}|x\right)=\frac{\varepsilon(x)BD(-1)}{p^\alpha}\sum_{p \nmid y, 1-y,1-xy} B(y)\overline{B}D(1-y)\overline{A}(1-xy).
		\end{align}
		The following map
		\begin{align*}
			\{y\in\mathbb{Z}_{p^\alpha}:p\nmid y,1-y,1-xy\}&\rightarrow\{z\in\mathbb{Z}_{p^\alpha}:p\nmid z,1-z,1-x+xz\}\\
			y&\mapsto y(1-x)(1-xy)^{-1}
		\end{align*}
		is a bijection. Hence, using the substitution $y\mapsto y(1-x)(1-xy)^{-1}$ in the sum in \eqref{6.9} yields
		\begin{align}\label{eq02}
			&{_{2}}F_{1}\left(\begin{array}{cc}A, & B\\ & D\end{array}|x\right)\notag\\
			&=\frac{\varepsilon(x)BD(-1)}{p^\alpha}\sum_{\substack{p\nmid z,1-z,\\1-x+xz}}[B(z(1-x+xz)^{-1})\overline{B}D((1-x)(1-z)(1-x+xz)^{-1})\notag\\
			&\hspace{3cm}\times \overline{A}((1-x)(1-x+xz)^{-1})]\notag  \\
			&=\frac{\overline{AB}D(1-x)\varepsilon(x)BD(-1)}{p^\alpha}\sum_{\substack{p\nmid z,1-z,\\1-x+xz}}B(z)\overline{B}D(1-z)A\overline{D}(1-x+xz).\notag\\
			&=\frac{\overline{B}(1-x)\varepsilon(x)BD(-1)}{p^\alpha}\sum_{\substack{p\nmid z,1-z,\\1-x+xz}}B(z)\overline{B}D(1-z)A\overline{D}(1-xz(x-1)^{-1}).
		\end{align}
		Note that we have assumed $p\nmid x-1$, so $p\mid x$  if and only if $p\mid x(x-1)^{-1}$. Therefore, we have $\varepsilon(x)=\varepsilon\left(x(x-1)^{-1}\right)$. Thus, replacing $\varepsilon(x)$ by $\varepsilon(x(x-1)^{-1})$ in \eqref{eq02} and then 
		using Definition \ref{def01} in the same, we conclude \eqref{arty}. \\ 
		Now, using Lemma \ref{lemahyp} and \eqref{arty} we find that
		\begin{align}\label{artu}
			{_{3}}F_{2}\left(\begin{array}{ccc}A, & B, & C\\ & D, & E\end{array}| 1\right)&=\frac{CE(-1)}{p^\alpha}\sum_{p\nmid y, 1-y} {_{2}}F_{1}\left(\begin{array}{cc}A, & B\\ & D\end{array}|y\right)C(y)\overline{C}E(1-y)\notag \\
			&=\frac{CE(-1)}{p^\alpha}\sum_{p\nmid y,1-y}\Bigg[  {_{2}}F_{1}\left(\begin{array}{cc}\overline{A}D, & B\\ & D\end{array}|y(y-1)^{-1}\right)\notag\\
			&\hspace{3cm}\times C(y)\overline{B}\overline{C}E(1-y)\Bigg].
		\end{align}
		It is easy to see that
		\begin{align*}
			\{y\in\mathbb{Z}_{p^\alpha}:p\nmid y,1-y\}&\rightarrow\{z\in\mathbb{Z}_{p^\alpha}:p\nmid z,1-z\}\\
			y&\mapsto y(y-1)^{-1}
		\end{align*}
		is a bijection, and hence, \eqref{artu} together with Lemma \ref{lemahyp} yields
		\begin{align*}
			{_{3}}F_{2}\left(\begin{array}{ccc}A, & B, & C\\ & D, & E\end{array}| 1\right)&=\frac{CE(-1)}{p^\alpha}\sum_{p\nmid z,1-z}\Bigg[{_{2}}F_{1}\left(\begin{array}{cc}\overline{A}D, & B\\ & D\end{array}|z\right)\\
			&\hspace{3cm}\times C\left( z(z-1)^{-1}\right) \overline{BC}E\left(-(z-1)^{-1}\right)\Bigg]\\
			&=\frac{E(-1)}{p^\alpha}\sum_{p\nmid z,1-z}C(z)B\overline{E}(1-z)	{_{2}}F_{1}\left(\begin{array}{cc}\overline{A}D, & B\\ & D\end{array}| z\right)\\
			&=B(-1)	{_{3}}F_{2}\left(\begin{array}{ccc}\overline{A}D, & B, & C\\ & D, & BC\overline{E}\end{array}| 1\right),
		\end{align*}
		concluding the proof of the lemma.
	\end{proof}
	The following lemma is the $\mathbb{Z}_{p^\alpha}$-analogue of (4.26) in \cite{greene}.
	\begin{lemma}\label{lem007}
		Let $p$ be an odd prime and let $\alpha\geq 1$ be an integer. Let $A,B,C,D,E$ be Dirichlet characters mod $p^\alpha$. Then,
		$${_{3}}F_{2}\left(\begin{array}{ccc}A, & B, & C\\ & D, & E\end{array}| 1\right)=AB(-1){_{3}}F_{2}\left(\begin{array}{ccc}\overline{A}D, & \overline{B}D, & C\\ & D, & \overline{AB}DE\end{array}| 1\right).$$	
	\end{lemma}
	\begin{proof}
		Firstly, we show that if $x\in\mathbb{Z}_{p^\alpha}$ is such that $p\nmid 1-x$, then
		\begin{align}\label{arto}
			{_{2}}F_{1}\left(\begin{array}{cc}A, & B\\ & D\end{array}| x\right)&=D(-1)\overline{AB}D(1-x){_{2}}F_{1}\left(\begin{array}{cc}\overline{A}D, & \overline{B}D\\ & D\end{array}| x\right). 
		\end{align}
		To prove this, we assume that $x\in\mathbb{Z}_{p^\alpha}$ satisfying $p\nmid 1-x$. It is easy to see that the following map is a bijection.
		\begin{align*}
			\{y\in\mathbb{Z}_{p^\alpha}:p\nmid y,1-y,1-xy\}&\rightarrow\{z\in\mathbb{Z}_{p^\alpha}:p\nmid z,1-z,1-xz\}\\
			y&\mapsto (1-y)(1-xy)^{-1}
		\end{align*}
		We substitute $z=(1-y)(1-xy)^{-1}$ in the sum in Definition \ref{def01} to obtain
		\begin{align}\label{eq05}
			{_{2}}F_{1}\left(\begin{array}{cc}A, & B\\ & D\end{array}| x\right)&=\frac{\varepsilon(x)BD(-1)}{p^\alpha}\sum_{\substack{p\nmid z,1-z,\\1-xz}}[B((1-z)(1-xz)^{-1})\notag\\
			&\hspace{2.5cm}\times \overline{B}D(z(1-x)(1-xz)^{-1}) \overline{A}((1-x)(1-xz)^{-1})] \notag\\
			&=\frac{\varepsilon(x)BD(-1)}{p^\alpha}\overline{AB}D(1-x)\sum_{\substack{p\nmid z,1-z,\\1-xz}}\overline{B}D(z)B(1-z)A\overline{D}(1-xz).
		\end{align}
		As a result, Definition \ref{def01} and \eqref{eq05} yield \eqref{arto}. Now, using Lemma \ref{lemahyp} and \eqref{arto} we find that
		\begin{align*}
			&{_{3}}F_{2}\left(\begin{array}{ccc}A, & B, & C\\ & D, & E\end{array}| 1\right)\\
			&=\frac{CE(-1)}{p^\alpha}\sum_{p\nmid y,1-y}	{_{2}}F_{1}\left(\begin{array}{cc}A, & B\\ & D\end{array}| y\right)C(y)\overline{C}E(1-y)\\
			&=\frac{CDE(-1)}{p^\alpha}\sum_{p\nmid y,1-y}	{_{2}}F_{1}\left(\begin{array}{cc}\overline{A}D, & \overline{B}D\\ & D\end{array}| y\right)C(y)\overline{ABC}DE(1-y)\\
			&=AB(-1){_{3}}F_{2}\left(\begin{array}{ccc}\overline{A}D, & \overline{B}D, & C\\ & D, & \overline{AB}DE\end{array}| 1\right).
		\end{align*}
		This completes the proof of the lemma.
	\end{proof}
	In \cite{dawsey}, corresponding to each of the transformations from (3.15) to (3.21) listed therein, Dawsey and McCarthy associated a map. The purpose was to have a group action, which ultimately concluded that certain hypergeometric functions (over finite fields) would yield the same value. For a detailed account, one can refer to \cite{dawsey} or Lemma 2.8 in \cite{bhowmikpeisert}. Here, we do the same but for hypergeometric functions with Dirichlet characters as arguments. The following lemma looks essentially the same as Lemma 2.8 in \cite{bhowmikpeisert}, except that the hypergeometric functions here involve the Dirichlet characters as defined in this article.
	\begin{lemma}\label{dlemma1}
		Let $X=\{(t_1,t_2,t_3,t_4,t_5)\in\mathbb{Z}_4^5: t_1,t_2,t_3\neq 0,t_4,t_5;~t_1+t_2+t_3\neq t_4,t_5\}$. Define the functions $f_i:X\rightarrow X,~i\in\{1,2,\ldots,7\}$ in the following manner:
		\begin{align*}
			f_1(t_1,t_2,t_3,t_4,t_5)&=(t_2-t_4,t_1-t_4,t_3-t_4,-t_4,t_5-t_4),\\	
			f_{2}\left(t_{1}, t_{2}, t_{3}, t_{4}, t_{5}\right)&=\left(t_{1}, t_{1}-t_{4}, t_{1}-t_{5}, t_{1}-t_{2}, t_{1}-t_{3}\right),\\
			f_{3}\left(t_{1}, t_{2}, t_{3}, t_{4}, t_{5}\right)&=\left(t_{2}-t_{4}, t_{2}, t_{2}-t_{5}, t_{2}-t_1,t_2-t_3\right),\\
			f_{4}\left(t_{1}, t_{2}, t_{3}, t_{4}, t_{5}\right)&=\left(t_{1}, t_{2}, t_{5}-t_{3}, t_{1}+t_{2}-t_{4}, t_{5}\right),\\
			f_{5}\left(t_1, t_{2}, t_{3}, t_{4}, t_{5}\right)&=\left(t_{1}, t_{4}-t_{2}, t_{3}, t_{4}, t_{1}+t_{3}-t_{5}\right),\\
			f_{6}\left(t_{1}, t_{2}, t_{3}, t_{4}, t_{5}\right)&=\left(t_{4}-t_{1}, t_{2}, t_{3}, t_{4}, t_{2}+t_{3}-t_{5}\right),\\
			f_{7}\left(t_{1},t_{2}, t_{3},t_{4}, t_{5}\right)&=\left(t_{4}-t_{1},t_{4}-t_{2},t_{3}, t_{4}, t_{4}+t_{5}-t_{1}-t_{2}\right).
		\end{align*}
		Then the group generated by $f_1,\ldots,f_7$, with operation composition of functions, is the set
		$$\mathcal{F}=\{f_0,f_i,f_j \circ f_l,f_4\circ f_1,f_6\circ f_2,f_5\circ f_3,f_1\circ f_4\circ f_1: 1\leq i\leq 7,~1\leq j\leq 3,~4\leq l\leq 7\},$$
		where $f_0$ is the identity map.
		Moreover, the group $\mathcal{F}$ acts on the set $X$.\\ Now, let $p\equiv 1\pmod 8$ be a prime and $\alpha$ be a positive integer. Let $\chi_4$ be a Dirichlet character mod $p^\alpha$ of order $4$. If we associate the $5$-tuple $(t_1,t_2,\ldots,t_5)\in X$ to the hypergeometric function ${_{3}}F_{2}\left(\begin{array}{ccc}\chi_4^{t_1}, & \chi_4^{t_2}, & \chi_4^{t_3}\\ & \chi_4^{t_4}, & \chi_4^{t_5}\end{array}| 1\right)$, 
		then each orbit of the group action consists of a number of $5$-tuples $(t_1,t_2,\ldots,t_5)$, and the corresponding ${}_3 F_{2}$ terms have the same value.
	\end{lemma}
	\begin{proof}
		The proof follows using Lemmas \ref{lem001} to \ref{lem007}. For example, the transformation in Lemma \ref{lem001} gives that
		$${_{3}}F_{2}\left(\begin{array}{ccc}\chi_4^{t_1}, & \chi_4^{t_2}, & \chi_4^{t_3}\\ & \chi_4^{t_4}, & \chi_4^{t_5}\end{array}| 1\right)={_{3}}F_{2}\left(\begin{array}{ccc}\chi_4^{t_2-t_4}, &\chi_4^{t_1-t_4} , & \chi_4^{t_3-t_4}\\ & \chi_4^{-t_4}, &\chi_4^{t_5-t_4}\end{array}| 1\right),$$
		and hence, it induces a map $f_1: X\rightarrow X$
		given by
		$$f_1(t_1,t_2,t_3,t_4,t_5)=(t_2-t_4,t_1-t_4,t_3-t_4,-t_4,t_5-t_4).$$
	\end{proof}
	\section{Proof of Theorem \ref{thm2}}
	We are now ready to prove Theorem \ref{thm2}. Recall that $\mathbb{Z}_{p^\alpha}^\ast=\langle g\rangle$ and $h=1-\chi_4(g)$. Since $-1\in\langle g^4\rangle$, we have $\chi_4(-1)=1$. Let $H=\langle g^4\rangle\cup g\langle g^4\rangle$ and $H_{ind}$ be the subgraph of $G^\ast(p^{\alpha})$ induced by $H$. Let us denote by $Re(z)$ the real part of the complex number $z$. As before, $x^{-1}$ denotes the multiplicative inverse of $x\in\mathbb{Z}_n^\ast$.
	\begin{proof}[Proof of Theorem \ref{thm2}]
		Since 
		$G^\ast(p^\alpha)$ is vertex-transitive, so we find that
		\begin{align}\label{tt}
			k_4(G^\ast(p^\alpha))
			&=\frac{p^\alpha}{4}\times \text{ number of cliques of order 4 in $G^\ast(p^\alpha)$ containing the vertex }0\notag \\
			&=\frac{p^\alpha}{4}\times k_3(H_{ind}).
		\end{align}
		So, our task is to find $k_3(H_{ind})$. We proceed as in the proof of Theorem 1.2 in \cite{bhowmikpeisert}.
		Let us denote by $k_3(H_{ind}, x)$ the number of triangles in $H_{ind}$ containing the vertex $x$. Let $a, b\in H$ be such that $\chi_4(ab^{-1})=1$. Then the map on the vertex set of $H_{ind}$ defined as
		\begin{align*}
			V(H_{ind})&\rightarrow V(H_{ind})\\
			x&\mapsto ba^{-1}x
		\end{align*}
		is a graph automorphism sending $a$ to $b$.
		Therefore, if $a, b\in H$ are such that $\chi_4(ab^{-1})=1$, then 
		\begin{align}\label{cond}
			k_3(H_{ind},a)=k_3(H_{ind},b).	
		\end{align} 
		Let $\langle g^4\rangle =\{x_1,\ldots,x_{p^{\alpha-1}\left(\frac{p-1}{4}\right)} \}$ with $x_1=1$ and $g\langle g^4\rangle=\{y_1,\ldots, y_{p^{\alpha-1}\left(\frac{p-1}{4}\right)}\}$ with $y_1=g$. Then,
		\begin{align}\label{pick}
			\sum_{i=1}^{p^{\alpha-1}\left(\frac{p-1}{4}\right)}k_3(H_{ind},x_i)+\sum_{i=1}^{p^{\alpha-1}\left(\frac{p-1}{4}\right)}k_3(H_{ind},y_i)=3\times k_3(H_{ind}).
		\end{align}
		By $\eqref{cond}$, we have  
		$$k_3(H_{ind},x_1)=k_3(H_{ind},x_2)=\cdots=k_3(H_{ind},x_{p^{\alpha-1}\left(\frac{p-1}{4}\right)})$$
		and 
		$$k_3(H_{ind},y_1)=k_3(H_{ind},y_2)=\cdots=k_3(H_{ind},y_{p^{\alpha-1}\left(\frac{p-1}{4}\right)}).$$
		Hence, \eqref{pick} yields 
		\begin{align}\label{1g}
			k_3(H_{ind})=\frac{p^{\alpha-1}(p-1)}{12}[k_3(H_{ind}, 1)+ k_3(H_{ind}, g)].
		\end{align}
		Thus, we need to find only $k_3(H_{ind}, 1)$ and $k_3(H_{ind}, g)$. We first find $k_3(H_{ind}, 1)$. 
		\par Employing \eqref{war}, we have 
		\begin{align}\label{xandy}
			&k_3(H_{ind},1)\notag \\
			&=\frac{1}{2\times 4^5}\sum_{\substack{p\nmid x,\\ 1-x}}[ (2+h\chi_4(1-x)+\overline{h}\overline{\chi_4}(1-x))(2+h\chi_4(x)+\overline{h}\overline{\chi_4}(x))]\notag\\
			&\hspace{1.5cm} \sum_{\substack{p\nmid y,1-y,\\x-y}}[(2+h\chi_4(1-y)+\overline{h}\overline{\chi_4}(1-y))
			(2+h\chi_4(x-y)+\overline{h}\overline{\chi_4}(x-y)) \notag \\
			&\hspace{2.5cm}\times  (2+h\chi_4(y)+\overline{h}\overline{\chi_4}(y))]. 
		\end{align}
		Let $i_1,i_2,i_3\in\{\pm 1\} $ and let $F_{i_1,i_2,i_3}$ denote the term $\chi_4^{i_1}(y)\chi_4^{i_2}(1-y)\chi_4^{i_3}(x-y)$. Next, we expand and evaluate the inner summation in \eqref{xandy}. We have 
		\begin{align}\label{sun}
			&\sum_{\substack{p\nmid y,1-y,\\x-y}}\hspace{-0.2cm}[2+h\chi_4(y)+\overline{h}\overline{\chi_4}(y)][2+h\chi_4(1-y)+\overline{h}\overline{\chi_4}(1-y)][2+h\chi_4(x-y)+\overline{h}\overline{\chi_4}(x-y)]\notag\\
			&=\sum_{\substack{p\nmid y,1-y,\\x-y}}[8+4h\chi_4(y)+4\overline{h}\overline{\chi_4}(y)+4h\chi_4(1-y)+4\overline{h}\overline{\chi_4}(1-y)+4h\chi_4(x-y)\notag\\&+4\overline{h}\overline{\chi_4}(x-y)
			+4\chi_4(y)\overline{\chi_4}(1-y)+4\overline{\chi_4}(y)\chi_4(1-y)+4\chi_4(y)\overline{\chi_4}(x-y)\notag\\
			&+4\overline{\chi_4}(y)\chi_4(x-y)
			+4\chi_4(1-y)\overline{\chi_4}(x-y)+4\overline{\chi_4}(1-y)\chi_4(x-y)\notag\\
			&+2h^2\chi_4(y)\chi_4(1-y)+2{\overline{h}}^2\overline{\chi_4}(y)\overline{\chi_4}(1-y)+2h^2\chi_4(y)\chi_4(x-y)\notag\\
			&+2{\overline{h}}^2\overline{\chi_4}(y)\overline{\chi_4}(x-y)
			+2h^2\chi_4(1-y)\chi_4(x-y)+2{\overline{h}}^2\overline{\chi_4}(1-y)\overline{\chi_4}(x-y)\notag\\
			&+h^3 F_{1,1,1}+2hF_{1,1,-1}+2hF_{1,-1,1}+2\overline{h}F_{1,-1,-1}+2hF_{-1,1,1}+2\overline{h}F_{-1,1,-1}
			\notag\\
			&+2\overline{h}F_{-1,-1,1}+{\overline{h}}^3F_{-1,-1,-1}].
		\end{align}
		Now, referring to Lemma \ref{lemsec1}, $\eqref{sun}$ yields
		\begin{align}\label{yonly}
			&\sum_{\substack{p\nmid y,1-y,\\x-y}}[2+h\chi_4(y)+\overline{h}\overline{\chi_4}(y)][2+h\chi_4(1-y)+\overline{h}\overline{\chi_4}(1-y)][2+h\chi_4(x-y)+\overline{h}\overline{\chi_4}(x-y)]\notag \\
			&=A+B\chi_4(x)+\overline{B}\overline{\chi_4}(x)+B\chi_4(x-1)+\overline{B}\overline{\chi_4}(x-1)-4p^{\alpha-1}\chi_4(x)\overline{\chi_4}(x-1)\notag \\
			&-4p^{\alpha-1}\overline{\chi_4}(x)\chi_4(x-1)-2h^2 p^{\alpha-1}\chi_4(x)\chi_4(x-1)-2{\overline{h}}^2 p^{\alpha-1}\overline{\chi_4}(x)\overline{\chi_4}(x-1)+C\varphi(x)\notag \\
			&+C\varphi(x-1)\notag\\
			&+\sum_{\substack{p\nmid y,1-y,\\x-y}}[h^3 F_{1,1,1}+2hF_{1,1,-1}+2hF_{1,-1,1}+2\overline{h}F_{1,-1,-1}+2hF_{-1,1,1}\notag
			\\&+2\overline{h}F_{-1,1,-1}+2\overline{h}F_{-1,-1,1}+{\overline{h}}^3F_{-1,-1,-1}]\notag\\
			&=:\mathcal{I},
		\end{align}
		where $A:=8p^{\alpha-1}(p-8)+4 Re(h^2\rho)$, $B:=-12hp^{\alpha-1}$ and $C:=4 Re(h^2\rho)$.
		\par Next, we introduce some notations. Let 
		\begin{align*}
			A_1&:=32(p-15)p^{\alpha-1}+16 Re(h^2\rho),\\
			B_1&:=16(p-15)hp^{\alpha-1}+16 Re(h^2\rho),\\
			C_1&:=16 Re(h^2\rho),\\
			D_1&:=8h Re(h^2\rho),\\
			E_1&:=8(p-15)h^2 p^{\alpha-1}+(4h^2+16) Re(h^2\rho),\text{ and}\\
			F_1&:=16(p-15)p^{\alpha-1}+8 Re(h^2\rho).
		\end{align*} 
		For $i\in\{1,2,3,4\}$ and $j\in\{1,2,\ldots,8\}$, we define the following character sums.
		\begin{align*}
			T_j&:=\sum_{p\nmid x,1-x}\sum_y \chi_4^{i_1}(y)\chi_4^{i_2}(1-y)\chi_4^{i_3}(x-y),\\
			U_{ij}&:=\sum_{p\nmid x,1-x}\chi_4^l(m)\sum_y \chi_4^{i_1}(y)\chi_4^{i_2}(1-y)\chi_4^{i_3}(x-y),\\
			V_{ij}&:=\sum_x\chi_4^{l_1}(x)\chi_4^{l_2}(1-x)\sum_y \chi_4^{i_1}(y)\chi_4^{i_2}(1-y)\chi_4^{i_3}(x-y),
		\end{align*}
		where 
		\begin{align*}
			l = \left\{
			\begin{array}{lll}
				1, & \hbox{if $i$ is odd,} \\
				-1, & \hbox{\text{otherwise};}
			\end{array}
			\right.
		\end{align*}
		\begin{align*}
			m = \left\{
			\begin{array}{lll}
				x, & \hbox{if $i\in\{1,2\}$,} \\
				1-x, & \hbox{\text{otherwise;}}
			\end{array}
			\right.
		\end{align*} 
		and 
		\begin{align*}
			(l_1,l_2) = \left\{
			\begin{array}{lll}
				(1,1), & \hbox{if $i=1$,} \\
				(1,-1), & \hbox{if $i=2$,} \\
				(-1,1), & \hbox{if $i=3$,} \\
				(-1,-1), & \hbox{if $i=4$.}
			\end{array}
			\right.
		\end{align*} 
		Also, corresponding to each $j$, let $(i_1,i_2,i_3)$ take the value according to the following:
		\begin{align*}
			(i_1, i_2, i_3) = \left\{
			\begin{array}{lll}
				(1, 1, 1), & \hbox{if $j=1$,} \\
				(1, 1, -1), & \hbox{if $j=2$,} \\
				(1, -1, 1), & \hbox{if $j=3$,} \\
				(1, -1, -1), & \hbox{if $j=4$,} \\
				(-1, 1, 1), & \hbox{if $j=5$,} \\
				(-1, 1, -1), & \hbox{if $j=6$,} \\
				(-1, -1, 1), & \hbox{if $j=7$,} \\
				(-1, -1, -1), & \hbox{if $j=8$.} \\
			\end{array}
			\right.
		\end{align*} 
		Then, using $\eqref{yonly}$ and the notations we just described, $\eqref{xandy}$ yields
		\begin{align*}
			&k_3(H_{ind},1)=\frac{1}{2048}\hspace{-0.1cm}\sum_{p\nmid x,1-x}\hspace{-0.2cm}[(2+h\chi_4(x)+\overline{h}\overline{\chi_4}(x))(2+h\chi_4(1-x)+\overline{h}\overline{\chi_4}(1-x))\times \mathcal{I}]\\
			=&\frac{1}{2048}\sum_{p\nmid x,1-x}( A_1+B_1\chi_4(x)+\overline{B_1}\overline{\chi_4}(x)+B_1\chi_4(x-1)+\overline{B_1}\overline{\chi_4}(x-1)\\
			&+C_1\varphi(x)+C_1\varphi(x-1)+D_1\chi_4(x)\varphi(x-1)+\overline{D_1}\overline{\chi_4}(x)\varphi(x-1)\\
			&+D_1\varphi(x)\chi_4(x-1)+\overline{D_1}\varphi(x)\overline{\chi_4}(x-1)+E_1\chi_4(x)\chi_4(x-1)+\overline{E_1}\overline{\chi_4}(x)\overline{\chi_4}(x-1) \\
			&+F_1\chi_4(x)\overline{\chi_4}(1-x)+\overline{F_1}\overline{\chi_4}(x)\chi_4(x-1))\\
			&+\frac{1}{2048}( 4h^3T_1+8hT_2+8h T_3+8\overline{h}T_4+8h T_5+8\overline{h}T_6+8\overline{h}T_7+4{\overline{h}}^3 T_8\\ 
			&+2h^4  U_{11}+4h^2 U_{12}+4h^2 U_{13}+8 U_{14}+4h^2 U_{15}+8 U_{16}+8 U_{17}+4{\overline{h}}^2 U_{18}\\
			&+4h^2 U_{21} +8 U_{22}+8 U_{23}+4{\overline{h}}^2 U_{24}+8 U_{25}+4{\overline{h}}^2 U_{26}+4{\overline{h}}^2 U_{27}+2{\overline{h}}^4 U_{28}\\
			&+2h^4  U_{31}+4h^2 U_{32}+4h^2 U_{33}+8 U_{34}+4h^2 U_{35}+8 U_{36}+8 U_{37}+4{\overline{h}}^2 U_{38}\\
			&+4h^2 U_{41} +8 U_{42}+8 U_{43}+4{\overline{h}}^2 U_{44}+8 U_{45}+4{\overline{h}}^2 U_{46}+4{\overline{h}}^2 U_{47}+2{\overline{h}}^4 U_{48}\\
			&+h^5 V_{11}+2h^3 V_{12}+2h^3 V_{13}+4h V_{14}+2h^3 V_{15}+4h V_{16}+4h V_{17}+4\overline{h} V_{18}\\
			&+2h^3 V_{21}+4h V_{22}+4h V_{23}+4\overline{h}V_{24}+4h V_{25}+4\overline{h}V_{26}+4\overline{h}V_{27}+2{\overline{h}}^3 V_{28}\\
			&+2h^3 V_{31}+4h V_{32}+4h V_{33}+4\overline{h}V_{34}+4h V_{35}+4\overline{h}V_{36}+4\overline{h}V_{37}+2{\overline{h}}^3 V_{38}\\
			&+4h V_{41}+4\overline{h}V_{42}+4\overline{h}V_{43}+2{\overline{h}}^3 V_{44}+4\overline{h}V_{45}+2{\overline{h}}^3 V_{46}+2{\overline{h}}^3 V_{47}+{\overline{h}}^5 V_{48} ). 	
		\end{align*}
		Employing Lemmas \ref{lema1} and \ref{corr}, we find that
		\begin{align}\label{bigex}
			&k_3(H_{ind},1)=\frac{1}{2048}\left[16(p-9)p^{\alpha-1} Re(h^2\rho)+32 p^{2\alpha-2}(p^2-20p+81)\right.\notag\\
			&+2Re\{\rho(8h^2(p-17)p^{\alpha-1}+4(h^2+4)Re(h^2\rho))\} 
			+32 Re(h^2\rho)Re(\xi h)+16Re(h^2\rho^2)\notag\\
			&+h^5 V_{11}+2h^3 V_{12}+2h^3 V_{13}+4h V_{14}+2h^3 V_{15}+4h V_{16}+4h V_{17}+4\overline{h} V_{18}\notag \\
			&+2h^3 V_{21}+4h V_{22}+4h V_{23}+4\overline{h}V_{24}+4h V_{25}+4\overline{h}V_{26}+4\overline{h}V_{27}+2{\overline{h}}^3 V_{28}\notag \\
			&+2h^3 V_{31}+4h V_{32}+4h V_{33}+4\overline{h}V_{34}+4h V_{35}+4\overline{h}V_{36}+4\overline{h}V_{37}+2{\overline{h}}^3 V_{38}\notag \\
			&\left.+4h V_{41}+4\overline{h}V_{42}+4\overline{h}V_{43}+2{\overline{h}}^3 V_{44}+4\overline{h}V_{45}+2{\overline{h}}^3 V_{46}+2{\overline{h}}^3 V_{47}+{\overline{h}}^5 V_{48}\right]. 
		\end{align}
		Now, we convert each term of the form $V_{i j}$ $[i \in\{1,2,3,4\}, j\in\{1,2, \ldots, 8\}]$ into its equivalent $p^{2\alpha}\cdot {_{3}}F_{2}$ form. We use the notation $\ell(t_{1}, t_{2}, \ldots, t_{5})\in \mathbb{Z}_4^5$ 
		for the term $p^{2\alpha}\cdot {_{3}}F_{2}\left(\begin{array}{ccc}\chi_4^{t_{1}}, & \chi_4^{t_{2}}, & \chi_4^{t_{3}}\\ & \chi_4^{t_{4}}, & \chi_4^{t_{5}}\end{array}| 1\right)$.
		Then, $\eqref{bigex}$ yields 
		\begin{align}\label{bigexp}
			&k_3(H_{ind},1)=\frac{1}{2048}\left[16(p-9)p^{\alpha-1} Re(h^2\rho)+32 p^{2\alpha-2}(p^2-20p+81)\right.\notag\\
			&+2Re\{\rho(8h^2(p-17)p^{\alpha-1}+4(h^2+4)Re(h^2\rho))\} 
			+32 Re(h^2\rho)Re(\xi h)+16Re(h^2\rho^2)\notag\\
			&+h^{5}\ell(3,1,1,2,2)+2 h^{3}\ell(1,1,3,2,0)+2 h^{3}\ell(3,1,1,0,2)+4 h\ell(1,1,3,0,0)\notag \\
			&+2h^{3}\ell(3,3,1,0,2)+4 h\ell(1,3,3,0,0)+4 h\ell(3,3,1,2,2)+4 \overline{h}\ell(1,3,3,2,0) \notag\\
			&+2 h^{3}\ell(3,1,3,2,2)+4 h\ell(1,1,1,2,0)+4 h\ell(3,1,3,0,2)+4 \overline{h}\ell(1,1,1,0,0)\notag\\
			&+4 h\ell(3,3,3,0,2)+4 \overline{h}\ell(1,3,1,0,0)+4 \overline{h}\ell(3,3,3,2,2)+2 {\overline{h}}^{3}\ell(1,3,1,2,0)\notag \\
			&+2 h^{3}\ell(3,1,3,2,0)+4 h\ell(1,1,1,2,2)+4 h\ell(3,1,3,0,0)+4 \overline{h}\ell(1,1,1,0,2)\notag\\
			&+4 h\ell(3,3,3,0,0)+4 \overline{h}\ell(1,3,1,0,2)+4 \overline{h}\ell(3,3,3,2,0)+2 {\overline{h}}^{3}\ell(1,3,1,2,2)  \notag \\
			&+4 h\ell(3,1,1,2,0)+4 \overline{h}\ell(1,1,3,2,2)+4 \overline{h}\ell(3,1,1,0,0)+2{\overline{h}}^{3}\ell(1,1,3,0,2)\notag\\
			&\left. +4 \overline{h}\ell(3,3,1,0,0)+2 \overline{h}^{3}\ell(1,3,3,0,2)+ 2\overline{h}^{3}\ell(3,3,1,2,0)+{\overline{h}}^{5}\ell(1,3,3,2,2)\right].
		\end{align}
		Next, we list the terms $\ell(t_1,t_2,\ldots, t_5)$ in each orbit of the group action of $\mathcal{F}$ on $X$, and then group the corresponding terms in $\eqref{bigexp}$ together (this is possible due to Lemma \ref{dlemma1}).
		The orbit representatives $\ell(1,1,1,0,0)$, $\ell(3,3,3,0,0)$, $\ell(1,3,3,2,0)$, $\ell(3,1,1,2,0)$ and $\ell(1,1,3,0,0)$ are the ones whose orbits exhaust the hypergeometric terms in $\eqref{bigexp}$. 
		We denote the $p^{2\alpha}\cdot {_{3}}F_{2}$ terms corresponding to these orbit representatives as $M_1,M_2,\ldots,M_5$ respectively. Then, $\eqref{bigexp}$ yields
		\begin{align}\label{mex}
			&k_3(H_{ind},1)=\frac{1}{2048}
			\left[16(p-9)p^{\alpha-1} Re(h^2\rho)+32 p^{2\alpha-2}(p^2-20p+81)\right.\notag\\
			&+2Re\{\rho(8h^2(p-17)p^{\alpha-1}+4(h^2+4)Re(h^2\rho))\} 
			+32 Re(h^2\rho)Re(\xi h)+16Re(h^2\rho^2)\notag\\
			&+h^{5}M_4+2 h^{3}M_1+2 h^{3}M_1+4 hM_5	+2h^{3}M_1+4 hM_5+4 hM_1+4 \overline{h}M_3 \notag\\
			&+2 h^{3}M_4+4 hM_5+4 hM_2+4 \overline{h}M_1+4 hM_5+4 \overline{h}M_5+4 \overline{h}M_5+2 {\overline{h}}^{3}M_3\notag \\
			&+2 h^{3}M_4+4 hM_5+4 hM_5+4 \overline{h}M_5
			+4 hM_2+4 \overline{h}M_1+4 \overline{h}M_5+2 {\overline{h}}^{3}M_3  \notag \\
			&+4 hM_4+4 \overline{h}M_2+4 \overline{h}M_5+2{\overline{h}}^{3}M_2
			+4 \overline{h}M_5+2 \overline{h}^{3}M_2+\left. 2\overline{h}^{3}M_2+{\overline{h}}^{5}M_3\right].
		\end{align}
		Simplifying \eqref{mex}, we have the reduced expression of $k_3(H_{ind},1)$ as follows.
		\begin{align}\label{spec}
			k_3(H_{ind},1)&=\dfrac{1}{2048}[16(p-9)p^{\alpha-1} Re(h^2\rho)+32 p^{2\alpha-2}(p^2-20p+81)\notag\\
			&+2Re\{\rho(8h^2(p-17)p^{\alpha-1}+4(h^2+4)Re(h^2\rho))\} 
			+32 Re(h^2\rho)Re(\xi h)\notag\\
			&+16Re(h^2\rho^2)+8(1-\overline{h})M_1+8(1-{h})M_2-8h M_3-8\overline{h}M_4+48 M_5]	.
		\end{align}
		Returning back to $\eqref{1g}$, we are now left to calculate $k_3( H_{ind},g)$. Again, by employing \eqref{war}, we have 
		\begin{align}\label{wand}
			&k_3(H_{ind},g)\notag\\
			&=\frac{1}{2048}\sum_{\substack{p\nmid x,\\g-x}}\sum_{\substack{p\nmid y,g-y,\\x-y}}\left[ (2+h\chi_4(g-x)+\overline{h}\overline{\chi_4}(g-x)) (2+h\chi_4(g-y)+\overline{h}\overline{\chi_4}(g-y))\notag \right.\\
			&\times\left. (2+h\chi_4(x-y)+\overline{h}\overline{\chi_4}(x-y))(2+h\chi_4(x)+\overline{h}\overline{\chi_4}(x)) (2+h\chi_4(y)+\overline{h}\overline{\chi_4}(y))\right]. 
		\end{align}	
		Using the substitutions $Y=yg^{-1}$ and $X=xg^{-1}$, and then using the fact that $h\chi_4(g)=\overline{h}$, \eqref{wand} yields
		\begin{align*}
			&k_3(H_{ind},g)\\
			&=\frac{1}{2048}\sum_{\substack{p\nmid x,\\1-x}}\sum_{\substack{p\nmid y,1-y,\\x-y}}\left[ (2+\overline{h}\chi_4(1-x)+h\overline{\chi_4}(1-x)) (2+\overline{h}\chi_4(1-y)+h\overline{\chi_4}(1-y))\notag \right.\\
			&\times\left. (2+\overline{h}\chi_4(x-y)+h\overline{\chi_4}(x-y)) (2+\overline{h}\chi_4(x)+h\overline{\chi_4}(x))(2+\overline{h}\chi_4(y)+h\overline{\chi_4}(y))\right].
		\end{align*}
		Comparing this with $\eqref{xandy}$ we see that the expansion of the expression inside this summation will consist of the same summation terms as in $\eqref{xandy}$, except that the coefficient corresponding to each summation term in this case, will 
		become the complex conjugate of the corresponding coefficient of the same summation term in \eqref{xandy}. So, we proceed to evaluate $k_3(H_{ind},g)$ in the same manner as we did for $k_3(H_{ind},1)$ and find that for the step analogous to \eqref{yonly}, there is a change in the value of the constants $``A$'' and $``C"$: $Re(\overline{h}^2\rho)$ takes the place of $Re({h}^2\rho)$; the other coefficients remain unchanged except for complex conjugation. Eventually, we have that the expression for $k_3(H_{ind},g)$ can be written by replacing $Re({h}^2\rho)$ by $Re(\overline{h}^2\rho)$ and taking the complex conjugate of the coefficients of $\rho,\xi$ as well as the complex conjugate of the coefficients of the hypergeometric terms corresponding to $k_3(H_{ind},1)$ in \eqref{spec}. Precisely, we have
		\begin{align}\label{spec1}
			k_3(H_{ind},g)&=\dfrac{1}{2048}[16(p-9)p^{\alpha-1} Re(\overline{h}^2\rho)+32 p^{2\alpha-2}(p^2-20p+81)\notag\\
			&+2Re\{\rho(8\overline{h}^2(p-17)p^{\alpha-1}+4(\overline{h}^2+4)Re(\overline{h}^2\rho))\} 
			+32 Re(\overline{h}^2\rho)Re(\xi \overline{h})\notag\\
			&+16Re(\overline{h}^2\rho^2)+8(1-{h})M_1+8(1-\overline{h})M_2+8\overline{h} M_3-8{h}M_4+48 M_5].
		\end{align}
		Finally, using \eqref{spec} and \eqref{spec1} in \eqref{1g}, we have
		\begin{align*}
			k_3(H_{ind})=\frac{p^{\alpha-1}(p-1)}{768}&\left[2 p^{2\alpha-2}(p^2-20p+81)+2(Im\rho)^2+4 Im\rho\cdot Im\xi\right.\notag\\
			&\left.-Re(M_3)+3 M_5\right].	
		\end{align*}
		Substituting the above value in $\eqref{tt}$, we complete the proof of the theorem.
	\end{proof}

	\section*{Appendix: Python Code}
	The Python code that we used to verify Theorem \ref{thm2} numerically can be found in the following link:\\\\
	https://github.com/AnwitaB/cliques\textunderscore of\textunderscore order\textunderscore four\textunderscore in\textunderscore Peisert-like\textunderscore graph.git
	\\\\ For convenience, we have also provided the code below. In the code, we refer to the theorem for the notations of $\rho,\xi,M_3$ and $M_5$. The code takes a prime $p\equiv 1\pmod 8$ and a positive integer $r$ as inputs, and computes the number of cliques of order four in the Peisert-like graph $G^\ast(p^r)$, the Jacobi sums (denoted by $\rho$ and $\xi$), and the hypergeometric terms (denoted by $M_3$ and $M_5$).
	\vspace{0.3cm}\\
	{\tt
		from sympy.ntheory.factor\_ import totient\\
		from math import gcd\\
		import cmath\\
		import numpy as np\\
		\\
		\#the function below calculates the number of cliques of order four in the
		Peisert-like graph G$^\ast$(n) where n=p\^{}r\\
		def cliques\_four (n,H): \hspace{1cm}\#H is the connection set of the graph \\ \hspace*{0.2cm} b1=(int)(totient(n)/2) \\
		\hspace*{0.2cm} number=0\\ 
		\hspace*{0.2cm} \tt flag1, flag2, flag3, flag4, flag5, flag6=0,0,0,0,0,0\\
		\hspace*{0.2cm} temp1, temp2, temp3, temp4, temp5, temp6=0,0,0,0,0,0\\
		\\
		\#now, checking if each tuple (i,j,k,l) forms a clique\\
		\hspace*{0.3cm}for i in range (n):  \\     
		\hspace*{0.6cm}for j in range(i+1,n): \#checking if ij is an edge\\
		\hspace*{0.9cm}temp1, flag1=(i-j)\%n,0\\
		\hspace*{0.9cm}for m in range (b1):\\
		\hspace*{1.2cm}if temp1==H[m]:\\
		\hspace*{1.5cm}flag1=1\\
		\hspace*{1.5cm}break\\
		\hspace*{0.9cm}if flag1==0:\\
		\hspace*{1.2cm}continue\\
		\hspace*{0.9cm}for k in range(j+1,n):\hspace*{0.3cm} \#checking if ik and jk are edges\\
		\hspace*{1.2cm}temp2, temp3, flag2, flag3=(i-k)\%n, (j-k)\%n, 0, 0\\
		\hspace*{1.2cm}for m in range (b1):\\
		\hspace*{1.5cm}if temp2==H[m]:\\
		\hspace*{1.8cm}flag2=1\\
		\hspace*{1.8cm}break\\
		\hspace*{1.2cm}for m in range (b1):\\
		\hspace*{1.5cm}if temp3==H[m]:\\
		\hspace*{1.8cm}flag3=1\\
		\hspace*{1.8cm}break\\
		\hspace*{1.2cm}if flag2==0 or flag3==0:\\
		\hspace*{1.5cm}continue\\
		\hspace*{1.2cm}for l in range(k+1,n):  \#checking if il,jl,kl are edges\\
		\hspace*{1.5cm}temp4, temp5, temp6=(i-l)\%n, (j-l)\%n, (k-l)\%n\\ 
		\hspace*{1.5cm}flag4, flag5, flag6= 0, 0, 0\\
		\hspace*{1.5cm}for m in range (b1):\\
		\hspace*{1.8cm}if temp4==H[m]:\\
		\hspace*{2.1cm}flag4=1\\
		\hspace*{2.1cm}break\\
		\hspace*{1.5cm}for m in range (b1):\\
		\hspace*{1.8cm}if temp5==H[m]:\\
		\hspace*{2.1cm}flag5=1\\
		\hspace*{2.1cm}break\\
		\hspace*{1.5cm}for m in range (b1):\\
		\hspace*{1.8cm}if temp6==H[m]:\\
		\hspace*{2.1cm}flag6=1\\
		\hspace*{2.1cm}break\\
		\hspace*{1.5cm}if flag4==0 or flag5==0 or flag6==0:\\
		\hspace*{1.8cm}continue\\
		\hspace*{1.5cm}number=number+1  \,   \#counts the number of tuples (i,j,k,l)\\
		\hspace*{5cm}\#forming a clique\\
		\\
		\hspace*{0.3cm}print("The number of cliques of order four in the Peisert-like graph G$^\ast$(p\^{}r) is ",number)\\
		\hspace*{0.3cm}return 1\\
		\\
		\\
		def raised(k):      \#this returns the value of i\^{}k \\
		\hspace*{0.3cm}if (k\%4)==0:\\
		\hspace*{0.6cm}return 1\\
		\hspace*{0.3cm}elif (k\%4)==1:\\
		\hspace*{0.6cm}return complex(0,1)\\
		\hspace*{0.3cm}elif (k\%4)==2:\\
		\hspace*{0.6cm}return -1\\
		\hspace*{0.3cm}else:\\
		\hspace*{0.6cm}return complex(0,1)*(-1)\\
		\\
		\#the function below calculates the Jacobi sums rho:=J(chi\_4,chi\_4)\\
		\#and zi:=J(chi\_4,phi) where chi\_4(g)=i, a primitive fourth root of\\
		\#unity and phi is the quadratic character, and g is the generator\\
		\# of Z\_n\^{}*\\
		def jacobi\_sums(n,zn,a):     \\
		\\
		\hspace*{0.3cm}pos\_x, pos\_x1=0,0\\
		\hspace*{0.3cm}rho, zi=0,0\\
		\hspace*{0.3cm}for i in range(totient(n)):\\
		\hspace*{0.6cm}x=zn[i]\\
		\hspace*{0.6cm}x1=(1-x)\%n\\
		\hspace*{0.6cm}if gcd(x1,n)==1:\\
		\hspace*{0.9cm}for j in range(totient(n)):  \#finds pos\_x such that g\^{}pos\_x=x\\
		\hspace*{1.2cm}if a[j]==x:\\
		\hspace*{1.5cm}pos\_x=j\\
		\hspace*{1.5cm}break\\
		\hspace*{0.9cm}for j in range(totient(n)):   \#finds pos\_x1 such that \\
		\hspace*{1.2cm}if a[j]==x1:\hspace*{3cm}\#g\^{}pos\_x1=1-x\\
		\hspace*{1.5cm}pos\_x1=j\\
		\hspace*{1.5cm}break\\
		\hspace*{0.9cm}rho=rho+raised(pos\_x+pos\_x1)\\
		\hspace*{0.9cm}zi=zi+raised(pos\_x+2*pos\_x1)\\
		\hspace*{0.3cm}print("The Jacobi sum rho:=J(chi\_4,chi\_4) is ",rho)\\
		\hspace*{0.3cm}print("The Jacobi sum zi:=J(chi\_4,phi) is ",zi)\\
		\hspace*{0.3cm}return 1\\
		\\
		\\
		def hypergeom\_sums(n,zn,a):                  \#this function calculates the\\
		\hspace*{5cm} \#hypergeometric terms M\_3 and M\_5\\
		\hspace*{0.3cm}x,x1=0,0 \\
		\hspace*{0.3cm}pos\_x, pos\_x1=0,0\\
		\hspace*{0.3cm}sum3,sum5=0,0,\\
		\hspace*{0.3cm}temp=0\\
		\hspace*{0.3cm}pos\_y, pos\_y1, pos\_xy=0,0,0\\
		\# For calculating the hypergeometric terms, which are double\\
		\#summations, we assume that the outer summation is indexed by\\
		\#x and the inner summation is indexed by y\\
		\hspace*{0.3cm}for i in range(totient(n)):\\
		\hspace*{0.6cm}x=zn[i]\\
		\hspace*{0.6cm}x1=(1-x)\%n\\
		\hspace*{0.6cm}if gcd(x1,n)==1:\\
		\hspace*{0.9cm}for j in range(totient(n)):     \#finds pos\_x such that \\
		\hspace*{1.2cm}if a[j]==x:\hspace*{4cm}\#g\^{}pos\_x=x\\
		\hspace*{1.5cm}pos\_x=j\\
		\hspace*{1.5cm}break\\
		\hspace*{0.9cm}for j in range(totient(n)):     \#finds pos\_x1 such that\\
		\hspace*{1.2cm}if a[j]==x1:\hspace*{3cm}\#g\^{}pos\_x1=1-x\\
		\hspace*{1.5cm}pos\_x1=j\\
		\hspace*{1.5cm}break\\
		\hspace*{0.9cm}temp=raised(pos\_x+pos\_x1)      \hspace*{1.5cm} \#chi\_4(x(1-x))\\
		\hspace*{0.9cm}temp1, tempo1=0,0\\
		\hspace*{0.9cm}for k in range(totient(n)):\\
		\hspace*{1.2cm}y=zn[k]\\
		\hspace*{1.2cm}y1=(1-y)\%n\\
		\hspace*{1.2cm}xy=(x-y)\%n\\
		\hspace*{1.2cm}if (gcd(y1,n)!=1) or (gcd(xy,n)!=1):\\
		\hspace*{1.5cm}continue\\
		\hspace*{1.2cm}for l in range(totient(n)):   \#finds pos\_y such that\\
		\hspace*{1.5cm}if a[l]==y:\hspace*{3cm}\#g\^{}pos\_y=y\\
		\hspace*{1.8cm}pos\_y=l\\
		\hspace*{1.8cm}break\\
		\hspace*{1.2cm}for l in range(totient(n)):   \#finds pos\_y1 such that \\
		\hspace*{1.5cm}if a[l]==y1:\hspace*{3cm}\#g\^{}pos\_y1=1-y\\
		\hspace*{1.8cm}pos\_y1=l\\
		\hspace*{1.8cm}break\\
		\hspace*{1.2cm}for l in range(totient(n)):   \#finds pos\_xy such that\\
		\hspace*{1.5cm}if a[l]==xy:\hspace{3cm}\#g\^{}pos\_xy=x-y\\
		\hspace*{1.8cm}pos\_xy=l\\
		\hspace*{1.8cm}break\\
		\hspace*{1.2cm}temp1=temp1+raised(pos\_y+pos\_y1+pos\_xy)  \\                  \hspace*{7cm}\#chi\_4(y(1-y)(x-y)) \\
		\hspace*{1.2cm}tempo1=tempo1+raised(pos\_y)*np.conj(raised(pos\_y1+pos\_xy))\\
		\hspace*{3cm}\#chi\_4(y)overline({chi\_4(1-y)(x-y))}) for M\_5      \\
		\hspace*{0.9cm}temp1=np.conj(temp1)                                         \#overline({chi\_4(y(1-y)(x-y))})\hspace*{-0.1cm} for M\_3\\
		\hspace*{0.9cm}sum3=sum3+temp*temp1                                         \#calculates M\_3 which involves the\\
		\hspace*{3cm}\#sum  chi\_4(x(1-x))overline({chi\_4(y(1-y)(x-y))})\\
		\hspace*{0.9cm}sum5=sum5+temp*tempo1                                   \#calculates M\_5 which involves the \\ \hspace*{2cm}\#sum  chi\_4(x(1-x))chi\_4(y)overline({chi\_4((1-y)(x-y))})\\ \\
		\hspace*{0.3cm}print("The hypergeometric sum M\_3 is ",sum3)\\
		\hspace*{0.3cm}print("The hypergeometric sum M\_5 is ",sum5)\\
		\hspace*{0.3cm}return 1\\ \\
		def main():\\
		\hspace*{0.3cm}print("enter a prime p congruent to 1 modulo 8")\\
		\hspace*{0.3cm}p = int(input())\\
		\hspace*{0.3cm}print("enter a positive integer r")\\
		\hspace*{0.3cm}r = int(input())\\
		\hspace*{0.3cm}n=int(pow(p,r))\\
		\hspace*{0.3cm}zn=list()\\
		\hspace*{0.3cm}div=list()\\
		\hspace*{0.3cm}g = 0\\
		\\
		\hspace*{0.3cm}for i in range(1,n):\\
		\hspace*{0.6cm}if gcd(i,n)==1:\\
		\hspace*{0.9cm}zn.append(i)    \hspace*{1.5cm}          \#zn contains the elements of Z\_n\^{}*\\ \\
		\hspace*{0.3cm}for i in range(1, int(totient(n)/2)+1):\\
		\hspace*{0.6cm}if totient(n)\%i==0:\\
		\hspace*{0.9cm}div.append(i)         \hspace*{0.6cm}    \#div contains all the positive divisors\\
		\hspace*{0.3cm}ldiv=len(div) \hspace*{1.5cm}\#of phi(n), except phi(n)\\ \\
		\hspace*{0.3cm}for i in range(totient(n)):        \#this loop finds g, a generator \\    
		\hspace*{0.6cm}var=0\hspace*{4cm}\#of Z\_n\^{}*. Each element a1 in \\
		\hspace*{0.6cm}a1=zn[i]\hspace*{3.5cm}\#Z\_n\^{}* is considered, and if \\
		\hspace*{0.6cm}for d in range (ldiv):\hspace*{0.9cm}\#a1\^{}dd=1 in Z\_n\^{}* for some dd  \\
		\hspace*{0.9cm}dd=div[d]\hspace*{3cm}\#in div, then a1 is discarded \\
		\hspace*{0.9cm}if (pow(a1,dd)\%n)==1:\\
		\hspace*{1.2cm}var=1\\
		\hspace*{1.2cm}break\\
		\hspace*{0.6cm}if var==0:\\
		\hspace*{0.9cm}g=a1\\
		\hspace*{0.9cm}break\\
		\hspace*{0.3cm}g1=(g*g*g*g)\%n\\
		\\
		\hspace*{0.3cm}H=list()            \hspace*{1cm} \#H is the connection set of the graph G$^\ast$(p\^{}r)\\
		\hspace*{0.3cm}for i in range(1, int(totient(n)/4)+1):\\
		\hspace*{0.6cm}temp=1\\
		\hspace*{0.6cm}for j in range(1, i+1):\\
		\hspace*{0.9cm}temp=temp*g1\\
		\hspace*{0.6cm}H.append(temp\%n)    \#powers of g\^{}4, that is, elements of <g\^{}4>,\\
		\hspace*{3.5cm} \#are appended to H\\
		\hspace*{0.3cm}for i in range(int(totient(n)/4)):\\
		\hspace*{0.6cm}H.append((H[i]*g)\%n)  \#elements of g<g\^{}4> are appended to H\\
		\hspace*{0.3cm}a=list()\\
		\hspace*{0.3cm}for i in range(totient(n)):\\
		\hspace*{0.6cm}s=(int)(pow(g,i))\\
		\hspace*{0.6cm}a.append(s\%n)  \hspace*{0.3cm}     \#a stores all the powers of the generator g, \\
		\hspace*{3.7cm}\#that is, 1,g,g\^{}2,..,g\^{}(totient(n)-1)\\
		\hspace*{0.3cm}cliques\_four(n, H)\\
		\hspace*{0.3cm}jacobi\_sums(n,zn,a)\\
		\hspace*{0.3cm}hypergeom\_sums(n,zn,a)\\
		main()
	}

\end{document}